%% file: root.tex
\documentclass{amsart}

\usepackage{amssymb}
\input{macros}

\newtheorem{Theorem}{Theorem}[section]

\newtheorem{Lemma}[Theorem]{Lemma}
\newtheorem{Corollary}[Theorem]{Corollary}
\newtheorem{Hypothesis}[Theorem]{Hypothesis}

\newtheorem{Claim}{Claim}

\theoremstyle{definition}
\newtheorem{Definition}[Theorem]{Definition}
\newtheorem{Example}[Theorem]{Example}

\theoremstyle{remark}

\numberwithin{equation}{section}

\begin{document}

\title{Automorphisms of $K$-groups II}


\author{Paul Flavell}
\address{The School of Mathematics\\University of Birmingham\\Birmingham B15 2TT\\Great Britain}
\email{P.J.Flavell@bham.ac.uk}
\thanks{A considerable portion of this research was done whilst the author was in receipt
of a Leverhulme Research Project grant and during visits to the Mathematisches Seminar,
Christian-Albrechts-Universit\"{a}t, Kiel, Germany.
The author expresses his thanks to the Leverhulme Trust for their support and
to the Mathematisches Seminar for its hospitality.}


\date{}

\begin{abstract}
    This work is a continuation of \emph{Automorphisms of $K$-groups I}, P. Flavell, preprint.
    The main object of study is a finite $K$-group $G$ that admits an elementary
    abelian group $A$ acting coprimely. For certain group theoretic properties $\mathcal P$,
    we study the $AC_{G}(A)$-invariant $\mathcal P$-subgroups of $G$.
    A number of results of McBride, {\em Near solvable signalizer functors on finite groups,}
    J. Algebra {\bf 78}(1) (1982) 181-214 and
    {\em Nonsolvable signalizer functors on finite groups,}
    J. Algebra {\bf 78}(1) (1982) 215-238 are extended.

    One purpose of this work is to build a general theory of automorphisms,
    one of whose applications will be a new proof of the Nonsolvable Signalizer Functor Theorem.
    As an illustration, this work concludes with a new proof of a special case
    of that theorem due to Gorenstein and Lyons.
\end{abstract}

\maketitle
\input{intro}
\input{pc}

\input{ap}

\input{prel}
\input{p}

\input{nap}
\input{lglob}

\input{md}

\input{gor}

\input{bib}

\end{document}

%% file: macros.tex
\newcommand{\cz}[2]{C_{#1}(#2)}

\newcommand{\n}[2]{N_{#1}(#2)}

\newcommand{\op}[2]{O_{#1}(#2)}

\newcommand{\oupper}[2]{O^{#1}(#2)}

\newcommand{\compp}[1]{\operatorname{comp}(#1)}
\newcommand{\comp}[2]{\operatorname{comp}_{#1}(#2)}
\newcommand{\compsol}[1]{\comp{\mathrm{sol}}{#1}}
\newcommand{\compasol}[1]{\comp{A,\mathrm{sol}}{#1}}
\newcommand{\layerr}[1]{E(#1)}
\newcommand{\layer}[2]{E_{#1}(#2)}

\newcommand{\gfitt}[1]{F^{*}(#1)}
\newcommand{\sol}[1]{\operatorname{sol}(#1)}

\newcommand{\zz}[1]{Z(#1)}

\newcommand{\card}[1]{|\,#1\,|}

\newcommand{\syl}[2]{\operatorname{Syl}_{#1}(#2)}
\newcommand{\aut}[1]{\operatorname{Aut}(#1)}
\newcommand{\out}[1]{\operatorname{Out}(#1)}

\newcommand{\gen}[2]{\langle \;#1 \mid #2\; \rangle}
\newcommand{\listgen}[1]{\langle \;#1\; \rangle}
\newcommand{\hyp}[1]{\operatorname{\rm Hyp}(#1)}

\newcommand{\set}[2]{\{ \;#1 \mid #2\;\}}
\newcommand{\listset}[1]{\{ \,#1\, \}}

\newcommand{\ltwo}[1]{\operatorname{L}_{2}(#1)}
\newcommand{\uthree}[1]{\operatorname{U}_{3}(#1)}
\newcommand{\sz}[1]{\operatorname{Sz}(#1)}
\newcommand{\sym}[1]{\operatorname{Sym}(#1)}
\newcommand{\badfour}{\listset{\ltwo{2^{r}}, \ltwo{3^{r}}, \uthree{2^{r}}, \sz{2^{r}}}}

\newcommand{\isomorphic}{\cong}
\newcommand{\normal}{\,\unlhd\,}
\newcommand{\subnormal}{\,\unlhd\unlhd\,}
\newcommand{\characteristic}{\operatorname{char}}

\newcommand{\br}[1]{\overline{#1}}
\newcommand{\nonid}{^\#}

\newcommand{\fancyP}{\mathcal P} 

\newcommand{\ICiteLemmaPsix}{\cite[Lemma~3.6]{I}}
\newcommand{\ICiteLemmaPeight}{\cite[Lemma~3.8]{I}}
\newcommand{\ICiteCoprimeActionComm}{\cite[Coprime Action(a)]{I}}
\newcommand{\ICiteCoprimeActionProd}{\cite[Coprime Action(e)]{I}}
\newcommand{\ICiteCoprimeActionSZ}{\cite[Coprime Action(g)]{I}}
\newcommand{\ICiteTheoremKone}{\cite[Theorem~4.1]{I}}
\newcommand{\ICiteTheoremKoneC}{\cite[Theorem~4.1(c)]{I}}
\newcommand{\ICiteTheoremKfour}{\cite[Theorem~4.4]{I}}
\newcommand{\ICiteTheoremKfourA}{\cite[Theorem~4.4(a)]{I}}
\newcommand{\ICiteTheoremKfourC}{\cite[Theorem~4.4(c)]{I}}
\newcommand{\ICiteSectionAQS}{\cite[\S6]{I}}
\newcommand{\ICiteLemmaAQSfive}{\cite[Lemma~6.5]{I}}

\newcommand{\ICiteLemmaAQSeightB}{\cite[Lemma~6.8(b)]{I}}
\newcommand{\ICiteCorollaryAQSnine}{\cite[Corollary~6.9]{I}}
\newcommand{\ICiteLemmaAQStwelve}{\cite[Lemma~6.12]{I}}
\newcommand{\ICiteTheoremAQSfiveABandTheoremKone}{\cite[Theorem~6.5(a),(b) and Theorem~4.1]{I}}
\newcommand{\ICiteLemmasAQSsixANDAQSeight}{\cite[Lemmas~6.6 and 6.7]{I}}
\newcommand{\ICiteSectionL}{\cite[\S9]{I}}
\newcommand{\ICiteTheoremLone}{\cite[Theorem~9.1]{I}}
\newcommand{\ICiteLemmaGthree}{\cite[Lemma~8.3]{I}}

%% file: intro.tex
\section{Introduction}\label{intro}
This work is a continuation of \cite{I}.
Namely we consider an elementary abelian group $A$
acting coprimely on the $K$-group $G$.
The main focus is on how the $A\cz{G}{A}$-invariant subgroups
interact with each other and influence the global
structure of $G$.

A new theme introduced is to consider a group theoretic
property $\fancyP$ for which $G$ possesses a unique
maximal normal $\fancyP$-subgroup $\op{\fancyP}{G}$ and
a unique normal subgroup $\oupper{\fancyP}{G}$ that is minimal
subject to the quotient being a $\fancyP$-group.
This leads to the notions of $\fancyP$-component and $(A,\fancyP)$-component
which generalize the notions of $\operatorname{sol}$-component
and $(A,\operatorname{sol})$-component introduced in \cite{I}.
In that paper,
we considered how the $A$-components of an $A\cz{G}{A}$-invariant
subgroup $H$ of $G$ are related to the
$(A,\operatorname{sol})$-components of $G$.
In \S\ref{lglob} we shall develop a partial extension of that
theory to the $(A,\fancyP)$-components of $H$.

If in addition $\fancyP$ is closed under extensions,
it will be shown in \S\ref{p} that $G$ possesses a
unique maximal $A\cz{G}{A}$-invariant $\fancyP$-subgroup.
This generalizes a result of McBride \cite{McB1} who
proved it in the case $\fancyP = $ ``is solvable''.
McBride also introduced the notion of a {\em near $A$-solvable} group.
In \S\ref{nap} we shall extend that work,
introducing the notion of a {\em near $(A,\fancyP)$-group.}

The results of \S\ref{p} and \S\ref{nap} have applications to
the study of nonsolvable signalizer functors.
In \S\ref{md} we shall present a result of McBride \cite[Theorem~6.5]{McB2}.
We have taken the liberty of naming this result the {\em McBride Dichotomy}
since it establishes a fundamental dichotomy in the proof of the
Nonsolvable Signalizer Functor Theorem.
As a further application,
this paper concludes with a new proof of a special case of the
Nonsolvable Signalizer Functor Theorem due to Gorenstein and Lyons \cite{GL}.

%% file: pc.tex
\section{$\mathcal P$-components}
Throughout this section we assume the following.
\begin{Hypothesis}\label{pc:1}
    $\fancyP$ is a group theoretic property that satisfies:
    \begin{enumerate}
        \item[1.]   $\fancyP$ is subgroup and quotient closed.

        \item[2.]   If $G/M$ and $G/N$ are $\fancyP$-groups then so is $G/M \cap N$.

        \item[3.]   If $M$ and $N$ are normal $\fancyP$-subgroups of the group $G$
                    then so is $MN$.
    \end{enumerate}
\end{Hypothesis}

\noindent Some obvious examples being: $\fancyP$ = ``is soluble''; ``is nilpotent'';
``is trivial''; ``is of odd order''.

For any group $G$ we define:
\begin{align*}
    \op{\fancyP}{G} &=  \gen{ N \normal G }{ \mbox{$N$ is a $\fancyP$-group} } \\
    \intertext{and}
    \oupper{\fancyP}{G} &=  \bigcap\set{ N \normal G }{ \mbox{$G/N$ is a $\fancyP$-group} }.
\end{align*}

\noindent Then $\op{\fancyP}{G}$ is the unique maximal normal $\fancyP$-subgroup of $G$
and $\oupper{\fancyP}{G}$ is the unique smallest normal subgroup whose
quotient is a $\fancyP$-group.

\begin{Definition}\label{pc:2}
    A \emph{$\fancyP$-component of $G$} is a subgroup $K$ of $G$
    that satisfies \[
        \mbox{$K \subnormal G$, %
        $K/\op{\fancyP}{K}$ is quasisimple and $K = \oupper{\fancyP}{K}$.}
    \]
    The set of $\fancyP$-components of $G$ is denoted by \[
        \comp{\fancyP}{G}
    \]
    and we define \[
        \layer{\fancyP}{G} = \listgen{\comp{\fancyP}{G}}.
    \]
\end{Definition}

\begin{Lemma}\label{pc:3}
    Let $G$ be a group.
    \begin{enumerate}
        \item[(a)]  $\op{\fancyP}{G}$ contains every subnormal $\fancyP$-subgroup of $G$.

        \item[(b)]  If $N \subnormal G$ then $\op{\fancyP}{N} = N \cap \op{\fancyP}{G}$.

        \item[(c)]  If $H \leq G$ then $\op{\fancyP}{G} \cap H \leq \op{\fancyP}{H}$.

        \item[(d)]  If $G = MN$ with $M,N \normal G$ then
                    $\oupper{\fancyP}{G} = \oupper{\fancyP}{M}\oupper{\fancyP}{N}$.
    \end{enumerate}
\end{Lemma}
\begin{proof}
    (a). Let $N$ be a subnormal $\fancyP$-subgroup of $G$ and set $H = \listgen{N^{G}}$.
    If $H = G$ then since $N \subnormal G$ we have $N = G$ and the result is clear.
    Suppose $H \not= G$ then by induction $N \leq \op{\fancyP}{H}$.
    Now $\op{\fancyP}{H} \characteristic H \normal G$
    so $\op{\fancyP}{H} \leq \op{\fancyP}{G}$ and then $N \leq \op{\fancyP}{G}$.

    (b). This follows from (a) and the fact that $\fancyP$ is subgroup closed.

    (c). Because $\fancyP$ is subgroup closed.

    (d). Note that $\oupper{\fancyP}{M} \characteristic M \normal G$.
    Set $\br{G} = G/\oupper{\fancyP}{M}\oupper{\fancyP}{N}$.
    Then $\br{G} = \br{M}\br{N}$.
    Now $\br{M}$ is a $\fancyP$-group since it is a quotient
    of the $\fancyP$-group $M/\oupper{\fancyP}{M}$.
    Similarly so is $\br{N}$.
    Now $\br{M}, \br{N} \normal \br{G}$ whence $\br{G}$ is a $\fancyP$-group
    and so $\oupper{\fancyP}{G} \leq  \oupper{\fancyP}{M}\oupper{\fancyP}{N}$.

    Set $G^{*} = G/\oupper{\fancyP}{G}$.
    Then $M^{*}$ is a $\fancyP$-group so $\oupper{\fancyP}{M} \leq \oupper{\fancyP}{G}$.
    Similarly $\oupper{\fancyP}{N} \leq \oupper{\fancyP}{G}$,
    completing the proof.
\end{proof}
\begin{Lemma}\label{pc:4}
    Let $G$ be a group and suppose $K,L \in \comp{\fancyP}{G}$.
    \begin{enumerate}
        \item[(a)]  If $N \subnormal G$ then
                    $\comp{\fancyP}{N} \subseteq \comp{\fancyP}{G}$.

        \item[(b)]  If $H \leq G$ and $K \leq H$ then $K \in \comp{\fancyP}{H}$.

        \item[(c)]  $K$ is perfect and possesses a unique maximal normal subgroup,
                    namely $\zz{K \bmod \op{\fancyP}{K}}$.

        \item[(d)]  (Wielandt) Suppose $N \subnormal G$.
                    Then either $K \leq N$ or $[K,N] \leq \op{\fancyP}{K}$.

        \item[(e)]  If $K \leq L$ then $K = L$.

        \item[(f)]  Either \[
                        \mbox{$K = L$ or $[K,L] \leq \op{\fancyP}{K} \cap \op{\fancyP}{L}$.}
                    \]
                    In particular, $K$ and $L$ normalize each other.

        \item[(g)]  If $G_{1},\ldots,G_{n} \subnormal G$ and
                    $K \leq \listgen{G_{1},\ldots,G_{n}}$
                    then $K \leq G_{i}$ for some $i$.

        \item[(h)]  If $\mathcal C, \mathcal D \subseteq \comp{\fancyP}{G}$
                    satisfy $\listgen{\mathcal C} = \listgen{\mathcal D}$
                    then $\mathcal C = \mathcal D$.

        \item[(i)]  $[K, \op{\fancyP}{G}\sol{G}] \leq \op{\fancyP}{K}$.
                    In particular $\op{\fancyP}{G}\sol{G}$ normalizes
                    every $\fancyP$-component of $G$.

        \item[(j)]  Set $\br{G} = G/\op{\fancyP}{G}$.
                    The map \[
                        \mbox{$\comp{\fancyP}{G} \longrightarrow \compp{\br{G}}$ %
                        defined by $K \mapsto \br{K}$}
                    \]
                    is an injection.
                    If every $\fancyP$-group is solvable then it is a bijection.
    \end{enumerate}
\end{Lemma}
\begin{proof}
    (a),(b). These follow immediately from the definition of $\fancyP$-component.

    (c). Suppose that $N$ is a proper normal subgroup of $K$.
    Now $K/\oupper{\fancyP}{K}$ is quasisimple so either $N \leq \zz{K \bmod \op{\fancyP}{K}}$
    or $N$ maps onto $K/\oupper{\fancyP}{K}$.
    Assume the latter.
    Then $K = N\op{\fancyP}{K}$ so
    $K/N \isomorphic \op{\fancyP}{K}/\op{\fancyP}{K} \cap N$,
    hence $K/N$ is a $\fancyP$-group.
    Then $K = \oupper{\fancyP}{K} \leq N$ and $K = N$.
    Thus $\zz{K \bmod \op{\fancyP}{K}}$ is the unique maximal normal subgroup of $K$.
    Also $K/\op{\fancyP}{K}$ is perfect so $K'$ maps onto
    $K/\op{\fancyP}{K}$ and hence $K = K'$.

    (d). Set $M = \listgen{N^{G}}$.
    If $N = G$ then the conclusion is clear so assume $N \not= G$.
    Now $N \subnormal G$ so $M \not= G$.
    If $K \leq M$ then the conclusion follows by induction,
    so we may assume that $K \not\leq M$.

    Suppose that $K$ is not normal in $KM$.
    Now $K \subnormal KM$ so there exists $g \in KM$ such that
    $K$ and $K^{g}$ normalize each other but $K \not= K^{g}$.
    Set $T = K^{g}K$.
    Now $KM = K^{g}M$ so $T = K^{g}(T \cap M)$.
    Note that $K^{g} \normal T$ and $T \cap M \normal T$.
    Since $K \normal T$ and $K$ is perfect,
    we have $K = [K,T] = [K,K^{g}][K,T \cap M]$ and then
    $K = (K \cap K^{g})(K \cap T \cap M)$.
    But $K \cap K^{g}$ and $K \cap T \cap M$ are proper normal subgroups
    of $K$,
    contrary to $K$ having a unique maximal normal subgroup.
    We deduce that $K \normal KM$,
    hence \[
        [K,M] \leq K \cap M \leq \zz{K \bmod \op{\fancyP}{K}}
    \]
    and so $[K,M,K] \leq \op{\fancyP}{K}$.
    Similarly $[M,K,K] \leq \op{\fancyP}{K}$.
    Since $\op{\fancyP}{K} \normal KM$ and $K$ is perfect,
    the Three Subgroups Lemma implies that $[K,M] \leq \op{\fancyP}{K}$.
    Since $N \leq M$ we have $[K,N] \leq \op{\fancyP}{K}$.

    (e). We have $K \subnormal L$ so either $K = L$ or $K \leq \zz{L \bmod \op{\fancyP}{L}}$.
    Assume the latter.
    Now $K$ is perfect,
    whence $K \leq \op{\fancyP}{L}$ and $K$ is a $\fancyP$-group.
    This is not possible since $K = \oupper{\fancyP}{K}$.
    Hence $K = L$.

    (f). Assume that $K \not= L$.
    Then (c) implies $K \not\leq L$ and $L \not\leq K$.
    Two applications of (d) imply $[K,L] \leq \op{\fancyP}{K} \cap \op{\fancyP}{L}$.

    (g). Suppose $K \not\leq G_{i}$ for all $i$.
    Then (d) implies that $\listgen{G_{1},\ldots,G_{n}}$ normalizes $K$
    and centralizes $K/\op{\fancyP}{K}$.
    This is absurd since $K \leq \listgen{G_{1},\ldots,G_{n}}$
    and $K/\op{\fancyP}{K}$ is perfect.

    (h). Let $C \in \mathcal C$.
    By (g) there exists $D \in \mathcal D$ with $C \leq D$.
    Then (e) forces $C = D$,
    whence $\mathcal C \subseteq \mathcal D$.
    Similarly $\mathcal D \subseteq \mathcal C$.

    (i). Since $K$ is not a $\fancyP$-group and is perfect,
    we have $K \not\leq \op{\fancyP}{G}$ and $K \not\leq \sol{G}$.
    Apply (d).

    (j). Since $K \subnormal G$, (a) implies that
    $\op{\fancyP}{K} = \op{\fancyP}{G} \cap K$,
    whence $\br{K} \isomorphic K/\op{\fancyP}{K}$ and so $\br{K}$ is quasisimple.
    Thus $\br{K} \in \compp{\br{G}}$.
    Suppose that $\br{K} = \br{L}$.
    Then $K \leq L\op{\fancyP}{G}$.
    As $K$ is not a $\fancyP$-group,
    (g) implies $K \leq L$ and then (e) forces $K = L$.
    Hence the map is an injection.

    Suppose that every $\fancyP$-group is solvable and that
    $\br{C} \in \compp{\br{G}}$.
    Choose $D$ minimal subject to $D \subnormal G$ and $\br{D} = \br{C}$.
    Suppose that $\oupper{\fancyP}{D} \not= D$.
    Then $\br{\oupper{\fancyP}{D}} \leq \zz{\br{C}}$ whence $\br{C}/\zz{\br{C}}$
    is an image of $\br{D}/\br{\oupper{\fancyP}{D}}$,
    which is an image of $D/\oupper{\fancyP}{D}$.
    Thus $\br{C}/\zz{\br{C}}$ is a $\fancyP$-group.
    This is a contradiction since every $\fancyP$-group is solvable.
    Hence $\oupper{\fancyP}{D} = D$.
    As $D \subnormal G$ we have $\op{\fancyP}{D} = D \cap \op{\fancyP}{G}$
    so $D/\op{\fancyP}{D} \isomorphic \br{C}$,
    which is quasisimple.
    Thus $D \in \comp{\fancyP}{G}$.
\end{proof}

We remark that in (j) the extra condition to ensure that the map
is a bijection is needed. For example, let $\fancyP$ be the property
defined by $G$ is a $\fancyP$-group if and only if $G = \sol{G}\layerr{G}$
and every component of $G$ is isomorphic to $A_{5}$.
Now let $G = A_{5} \operatorname{wr} A_{5}$.

%% file: ap.tex
\section{$(A,\mathcal P)$-components}\label{ap}
Throughout this section,
assume the following.
\begin{Hypothesis}\label{ap:1}\mbox{}
    \begin{itemize}
        \item   Hypothesis~\ref{pc:1}.
        \item   $A$ is a finite group.
    \end{itemize}
\end{Hypothesis}

\begin{Definition}\label{ap:2}
    Let $G$ be a group on which $A$ acts.
    An \emph{$(A,\fancyP)$-component of $G$} is an $A$-invariant
    subgroup $K$ of $G$ that satisfies \[
        \mbox{$K \subnormal G$, $K/\op{\fancyP}{K}$ is $A$-quasisimple %
        and $K = \oupper{\fancyP}{K}$}.
    \]
    The set of $(A,\fancyP)$-components of $G$ is denoted by \[
        \comp{A,\fancyP}{G}.
    \]
\end{Definition}

\begin{Lemma}\label{ap:3}
    Let $G$ be a group on which $A$ acts.
    The $(A,\fancyP)$-components of $G$ are the subgroups
    generated by the orbits of $A$ on $\comp{\fancyP}{G}$.
    Distinct orbits generate distinct $(A,\fancyP)$-components.
\end{Lemma}
\begin{proof}
    Suppose $\listset{K_{1},\ldots,K_{n}}$ is an orbit of $A$ on $\comp{\fancyP}{G}$
    and define $K = \listgen{K_{1},\ldots,K_{n}}$.
    Certainly $K \subnormal G$.
    Lemma~\ref{pc:4}(f) implies $K_{i} \normal K$ for each $i$,
    so $K = K_{1}\cdots K_{n}$.
    By Lemma~\ref{pc:3}(d),
    $\oupper{\fancyP}{K} = \oupper{\fancyP}{K_{1}} \cdots \oupper{\fancyP}{K_{n}}%
    = K_{1} \cdots K_{n} = K$.
    Using Lemma~\ref{pc:4}(j),
    with $K$ in the role of $G$,
    we see that $K/\op{\fancyP}{K}$ is the central product of the quasisimple
    groups $K_{i}\op{\fancyP}{K}/\op{\fancyP}{K}$ and that these are
    permuted transitively by $A$.
    Thus $K/\op{\fancyP}{K}$ is $A$-quasisimple and hence
    $K$ is an $(A,\fancyP)$-component.

    Conversely suppose that $K \in \comp{A,\fancyP}{G}$.
    Set $\br{K} = K/\op{\fancyP}{K}$ so that
    $\br{K} = \br{K_{1}}*\cdots*\br{K_{n}}$ with each $\br{K_{i}}$ quasisimple
    and $A$ acting transitively on $\listset{\br{K_{1}},\ldots,\br{K_{n}}}$.
    Let $L_{i}$ be the inverse image of $\br{K_{i}}$ in $K$.
    Then $L_{i} \normal K$ and $K = L_{1}\cdots L_{n}$.
    Set $K_{i} = \oupper{\fancyP}{L_{i}} \subnormal G$.
    By Lemma~\ref{pc:3}(d) we have $K = \oupper{\fancyP}{K} = K_{1} \cdots K_{n}$.
    Then $K_{i}$ maps onto $\br{K_{i}}$ so $L_{i} = \op{\fancyP}{K}K_{i}$.
    Again by Lemma~\ref{pc:3}(d),
    $K_{i} = \oupper{\fancyP}{L_{i}} = %
    \oupper{\fancyP}{\op{\fancyP}{K}}\oupper{\fancyP}{K_{i}} %
    = \oupper{\fancyP}{K_{i}}$.
    Thus $K_{i} \in \comp{\fancyP}{G}$ and $K$ is the subgroup generated by
    an orbit of $A$ on $\comp{\fancyP}{G}$.
    Finally, Lemma~\ref{pc:4}(h) implies that distinct orbits generate
    distinct $(A,\fancyP)$-components.
\end{proof}

\begin{Lemma}\label{ap:4}
    Let $G$ be a group on which $A$ acts and suppose
    $K,L \in \comp{A,\fancyP}{G}$.
    \begin{enumerate}
        \item[(a)]  If $N$ is an $A$-invariant subnormal subgroup of $G$
                    then $\comp{A,\fancyP}{N} \subseteq \comp{A,\fancyP}{G}$.

        \item[(b)]  If $H$ is an $A$-invariant subgroup of $G$ and $K \leq H$
                    then $K \in \comp{A,\fancyP}{H}$.

        \item[(c)]  $K$ is perfect and possesses a unique maximal $A$-invariant
                    subnormal subgroup,
                    namely $\zz{K \bmod \op{\fancyP}{K}}$.

        \item[(d)]  Suppose $N$ is an $A$-invariant subnormal subgroup of $G$.
                    Then either $K \leq N$ or $[K,N] \leq \op{\fancyP}{K}$.

        \item[(e)]  If $K \leq L$ then $K = L$.

        \item[(f)]  Either \[
                        \mbox{$K = L$ or $[K,L] \leq \op{\fancyP}{K} \cap \op{\fancyP}{L}$.}
                    \]
                    In particular, $K$ and $L$ normalize each other.

        \item[(g)]  Suppose $G_{1},\ldots,G_{n}$ are $A$-invariant subnormal
                    subgroups of $G$ and $K \leq \listset{G_{1},\ldots,G_{n}}$
                    then $K \leq G_{i}$ for some $i$.

        \item[(h)]  Suppose $\mathcal C, \mathcal D \subseteq \comp{A,\fancyP}{G}$
                    satisfy $\listgen{\mathcal C} = \listgen{\mathcal D}$.
                    Then $\mathcal C = \mathcal D$.

        \item[(i)]  $[K,\op{\fancyP}{G}\sol{G}] \leq \op{\fancyP}{K}$.
                    In particular, $\op{\fancyP}{G}\sol{G}$ normalizes every
                        $(A,\fancyP)$-component of $G$.

        \item[(j)]  Set $\br{G} = G/\op{\fancyP}{G}$.
                    The map \[
                        \mbox{$\comp{A,\fancyP}{G} \longrightarrow \comp{A}{\br{G}}$ %
                        defined by $K \mapsto \br{K}$}
                    \]
                    is an injection.
                    If every $\fancyP$-group is solvable then it is a bijection.

        \item[(k)]  $K \normal \listgen{K^{G}}$.
    \end{enumerate}
\end{Lemma}
\begin{proof}
    (a) and (b) are immediate from the definitions.

    (c),(e),(f),(g),(h),(i),(j) follow with the same argument as used in the
    proof of Lemma~\ref{pc:4}.

    (d) follows from Lemmas~\ref{ap:3} and \ref{pc:4}(d).

    (k).\ By Lemma~\ref{ap:3} we have $K = \listgen{K_{1},\ldots,K_{n}}$
    where $\listset{K_{1},\ldots,K_{n}} \subseteq \comp{\fancyP}{G}$
    so Lemma~\ref{pc:4}(f) implies $K_{i} \normal \listgen{K^{G}}$.
    Then $K \normal \listgen{K^{G}}$.
\end{proof}

%% file: prel.tex
\section{Preliminaries}\label{prel}

\begin{Lemma}\label{prel:1}
    Let $r$ be a prime and $A$ and elementary abelian
    $r$-group that acts coprimely on the $K$-group $G$.
    Suppose $K \in \comp{A}{G}$ and that $H$ is an
    $A\cz{K}{A}$-invariant subgroup of $G$ with
    $H \cap K \leq \zz{K}$.
    Then $[H,K] = 1$.
\end{Lemma}
\begin{proof}
    Set $\br{G} = G/\zz{\layerr{G}}$.
    We have $[H,\cz{K}{A}] \leq H \cap \layerr{G}$ so as $K \normal \layerr{G}$
    it follows that $[H,\cz{K}{A},\cz{K}{A}] \leq H \cap K \leq \zz{K}$.
    Note that $\cz{\br{K}}{A} = \br{\cz{K}{A}}$ by Coprime Action.
    Then $[\br{H},\cz{\br{K}}{A},\cz{\br{K}}{A}] = 1$.
    The Three subgroups Lemma implies that $[\br{H},\cz{\br{K}}{A}'] = 1$.
    Then $H$ permutes the components of $\br{G}$ onto which $\cz{\br{K}}{A}'$
    projects nontrivially.
    By \ICiteTheoremKfourA,
    $\cz{\br{K}}{A}' \not= 1$ so these components are precisely
    the components of $\br{K}$.
    We deduce that $\br{H}$ normalizes $\br{K}$ and then
    that $H$ normalizes $K$.
    Then $[H,\cz{K}{A}] \leq H \cap K \leq \zz{K}$
    so $[\br{H},\cz{\br{K}}{A}] = 1$.
    \ICiteTheoremKfourC\ implies that $[\br{H},\br{K}] = 1$.
    Since $K$ is perfect,
    it follows from the Three Subgroups Lemma that $[H,K] = 1$.
\end{proof}

\begin{Lemma}\label{prel:2}
    Let $\fancyP$ be a group theoretic property that satisfies:
    \begin{enumerate}
        \item[(a)]  $\fancyP$ is subgroup and quotient closed.

        \item[(b)]  If $G/M$ and $G/N$ are $\fancyP$-groups
                    then so is $G/M \cap N$.
    \end{enumerate}
    Suppose the group $A$ acts coprimely on the group $G$,
    that $P$ is an $A$-invariant subgroup of $G$
    and that $K \in \comp{A}{G}$.
    Assume that \[
        \mbox{$\cz{K}{A} \leq \n{G}{P}$ %
        and $[P,\cz{K}{A}]$ is a $\fancyP$-group.}
    \]
    Then \[
        \mbox{$P \leq \n{G}{K}$ or $\cz{K/\zz{K}}{A}$ is a $\fancyP$-group.}
    \]
\end{Lemma}
\begin{proof}
    Set $M = \listgen{\cz{K}{A}^{P}} = [P,\cz{K}{A}]\cz{K}{A} \leq \layerr{G}$.
    We have $M = [P,\cz{K}{A}](M \cap K)$ and $M \cap K \normal M$
    since $K \normal \layerr{G}$.
    Now $M/M \cap K$ is a $\fancyP$-group since it is isomorphic
    to a quotient of $[P,\cz{K}{A}]$.
    Thus \[
        \oupper{\fancyP}{M} \leq M \cap K.
    \]
    Since $P \leq \n{G}{M}$ we have $P \leq \n{G}{\oupper{\fancyP}{M}}$.

    Set $E = \layerr{G}$ and $\br{G} = G/\zz{\layerr{G}}$,
    so $\br{E}$ is the direct product of the components of $\br{G}$.
    Set $N = \oupper{\fancyP}{M}$.
    Suppose that $\br{N} \not= 1$.
    Now $\br{N}$ is $A$-invariant so $P$ permutes the components of
    $\br{G}$ onto which $\br{N}$ projects nontrivially.
    Since $\br{N} \leq \br{K}$ and both $\br{N}$ and $\br{K}$
    are $A$-invariant,
    these components are precisely the components of $\br{K}$.
    Then $P$ normalizes $\br{K}$ and hence $K$.

    Suppose that $\br{N} = 1$.
    Then $N \leq \zz{\layerr{G}}$ and so $M/M \cap \zz{\layerr{G}}$
    is a $\fancyP$-group.
    As $\cz{K}{A} \leq M$ and $\zz{K} = K \cap \zz{\layerr{G}}$
    it follows that $\cz{K}{A}/\cz{K}{A} \cap \zz{K}$
    is a $\fancyP$-group.
    Since $A$ acts coprimely on $G$,
    the quotient is isomorphic to $\cz{K/\zz{K}}{A}$,
    completing the proof.
\end{proof}
    

%% file: p.tex
\section{$\mathcal P$-subgroups}\label{p}

\begin{Definition}\label{p:1}
    Let $A$ be a group that acts on the group $G$
    and let $\fancyP$ be a group theoretic property.
    Then \[
        \op{\fancyP}{G;A} = \gen{ P \leq G }%
        { \mbox{$\fancyP$ is an $A\cz{G}{A}$-invariant $\fancyP$-subgroup} }.
    \]
\end{Definition}
\noindent We are interested in situations where
$\op{\fancyP}{G;A}$ is a $\fancyP$-group,
in other words,
when does $G$ possess a unique maximal $A\cz{G}{A}$-invariant $\fancyP$-subgroup?
The goal of this section is to prove the following.

\begin{Theorem}\label{p:2}
    Let $\fancyP$ be a group theoretic property that is closed under
    subgroups, quotients and extensions.
    Let $A$ be an elementary abelian $r$-group for some prime $r$
    and assume that $A$ acts coprimely on the $K$-group $G$.
    Then $\op{\fancyP}{G;A}$ is a $\fancyP$-group.
\end{Theorem}
\noindent As an immediate consequence we have the following.
\begin{Corollary}\label{p:3}
    Let $r$ be a prime and $A$ an elementary abelian $r$-group
    that acts on the group $G$.
    Suppose that $\theta$ is an $A$-signalizer functor on $G$
    and that $\theta(a)$ is a $K$-group for all $a \in A\nonid$.

    Let $\fancyP$ be a group theoretic property that is closed under
    subgroups, quotients and extensions.
    Define $\theta_{\fancyP}$ by \[
        \theta_{\fancyP}(a) = \op{\fancyP}{\theta(a);A}
    \]
    for all $a \in A\nonid$.
    Then $\theta_{\fancyP}$ is an $A$-signalizer functor.
\end{Corollary}
\noindent This generalizes a result of McBride \cite[Lemma~3.1]{McB1},
who proves it in the case $\fancyP = \mbox{``is solvable''}$
and $\theta$ is near solvable.

Throughout the remainder of this section,
we assume the hypotheses of Theorem~\ref{p:2}.

\begin{Lemma}\label{p:4}
    Assume that $N$ is an $A\cz{G}{A}$-invariant subgroup of $G$.
    \begin{enumerate}
        \item[(a)]  $\cz{G}{A}$ normalizes $\op{\fancyP}{N;A}$.

        \item[(b)]  Suppose that $\op{\fancyP}{G;A}$ and $\op{\fancyP}{N;A}$
                    are $\fancyP$-groups.
                    Then \[
                        \op{\fancyP}{N;A} = \op{\fancyP}{G;A} \cap N.
                    \]
                    If in addition,
                    $N \normal G$ then $\op{\fancyP}{N;A} \normal \op{\fancyP}{G;A}$.

        \item[(c)]  Suppose $N = N_{1} \times\cdots\times N_{m}$ with
                    each $N_{i}$ being $A$-invariant.
                    Then \[
                        \op{\fancyP}{N;A} = %
                        \op{\fancyP}{N_{1};A} \times\cdots\times \op{\fancyP}{N_{m};A}.
                    \]
    \end{enumerate}
\end{Lemma}
\begin{proof}
    (a). Since $\cz{G}{A}$ normalizes $A\cz{N}{A}$,
    it permutes the $A\cz{N}{A}$-invariant subgroups of $N$.

    (b). By (a),
    $\op{\fancyP}{N;A} \leq \op{\fancyP}{G;A} \cap N$.
    Moreover, $\op{\fancyP}{G;A} \cap N$ is an $A\cz{N}{A}$-invariant
    $\fancyP$-subgroup of $N$,
    proving the reverse inclusion.

    (c). Trivial.
\end{proof}

\begin{Lemma}\label{p:5}
    Suppose that $N$ is an $A$-invariant normal subgroup of $G$
    and that \[
        N = N_{1} \times\cdots\times N_{m}
    \]
    with each $N_{i}$ being simple.
    For each $i$,
    let $\pi_{i}:N \longrightarrow N_{i}$
    be the projection map.
    Suppose also that $B$ is a subgroup of $A$ that normalizes
    but does not centralize each $N_{i}$.
    \begin{enumerate}
        \item[(a)]  $\cz{N}{A}\pi_{i} = \cz{N_{i}}{B}$ for each $i$.

        \item[(b)]  $\op{\fancyP}{N;A} = \op{\fancyP}{N;B}$.

        \item[(c)]  If $X$ is an $A\cz{N}{A}$-invariant subgroup of $G$
                    that normalizes each $N_{i}$ then \[
                        [X,\cz{N}{B}] \leq (X \cap N)\pi_{1} \times\cdots\times (X \cap N)\pi_{m}.
                    \]
    \end{enumerate}
\end{Lemma}
\begin{proof}
    Note that $A$ permutes $\listset{N_{1},\ldots,N_{m}}$
    since this is the set of components of $N$.
    For each $i$,
    set $A_{i} = \n{A}{N_{i}}$.
    Now $N_{i}$ is a simple $K$-group so \ICiteTheoremKone\
    implies that the Sylow $r$-subgroups of $\aut{K_{i}}$ are cyclic.
    Hence $A_{i} = \cz{A}{N_{i}}B$.
    Using \ICiteLemmaPsix\
    we have $\cz{N}{A}\pi_{i} = \cz{N_{i}}{A_{i}} = \cz{N_{i}}{B}$
    so (a) follows.

    (b). Choose $1 \leq i \leq m$,
    let $X$ be a $B\cz{N_{i}}{B}$-invariant $\fancyP$-subgroup of $N_{i}$
    and set $Y = \listgen{X^{A}}$.
    Since $A_{i} = \n{A}{N_{i}}$ it follows that $Y$ is the direct
    product of $\card{A:A_{i}}$ copies of $X$.
    Then $Y$ is an $A$-invariant $\fancyP$-subgroup.
    Now $\cz{N}{B} = \cz{N_{1}}{B} \times\cdots\times \cz{N_{m}}{B}$.
    It follows that $Y$ is $\cz{N}{B}$-invariant.
    Now $\cz{N}{A} \leq \cz{N}{B}$ whence $Y \leq \op{\fancyP}{N;A}$.
    Using Lemma~\ref{p:4}(c),
    with $B$ in place of $A$,
    we deduce that $\op{\fancyP}{N;B} \leq \op{\fancyP}{N;A}$.

    To prove the opposite containment,
    suppose that $Z$ is an $A\cz{N}{A}$-invariant $\fancyP$-subgroup of $N$.
    From (a) it follows that each $Z\pi_{i}$ is a
    $B\cz{N_{i}}{B}$-invariant of $N_{i}$,
    whence $Z \leq \op{\fancyP}{N_{1};B} \times\cdots\times \op{\fancyP}{N_{m};B}$.
    Another application of Lemma~\ref{p:4}(c)
    implies $Z \leq \op{\fancyP}{N;B}$,
    completing the proof.

    (c). Let $x \in X$ and $c \in \cz{N}{A}$.
    Then $c = (c\pi_{1})\cdots(c\pi_{m})$ and,
    as $x$ normalizes each $N_{i}$,
    it follows that $[x,c\pi_{i}] \in N_{i}$.
    Hence \[
        [x,c] = [x,c\pi_{1}] \cdots [x,c\pi_{m}].
    \]
    Now $[x,c] \in X \cap N$ so it follows that \[
        [x,c]\pi_{i} = [x,c\pi_{i}].
    \]
    In particular, $[x,c\pi_{i}] \in (X \cap N)\pi_{i}$.
    Since $\cz{N}{A}\pi_{i} = \cz{N_{i}}{B}$ we have
    $[x,\cz{N_{i}}{B}] \leq (X \cap N)\pi_{i}$.
    Now $\cz{N}{B} = \cz{N_{1}}{B} \times\cdots\times \cz{N_{m}}{B}$
    so the result follows.
\end{proof}

\begin{Lemma}\label{p:6}
    Suppose that $N$ is an $A$-invariant normal subgroup of $G$
    that is the direct product of nonabelian simple groups
    and that $X$ is an $A\cz{N}{A}$-invariant subgroup of $G$.
    Assume that $X$ and $\op{\fancyP}{N;A}$ are $\fancyP$-groups.
    Then $X \leq \n{G}{\op{\fancyP}{N;A}}$.
\end{Lemma}
\begin{proof}
    Assume false and consider a counterexample with $\card{A}$ minimal
    and then $\card{G}+\card{N}+\card{X}$ minimal.
    Then $G = XN$.
    By Coprime Action,
    $X = \cz{X}{A}[X,A] = \gen{ \cz{X}{B} }{ B \in \hyp{A} }$.
    Since $\cz{X}{A}$ normalizes $\op{\fancyP}{N;A}$ it follows
    that $X = [X,A] = \cz{X}{B}$ for some $B \in \hyp{A}$.
    We have \[
        N = N_{1} \times\cdots\times N_{m}
    \]
    where each $N_{i}$ is nonabelian and simple.
    Then $\listset{N_{1},\ldots,N_{m}}$ is the set of components
    of $N$ and is hence permuted by $AG$.
    Using Lemma~\ref{p:4}(c) and the minimality of $\card{N}$
    it follows that $AX$ is transitive on $\listset{N_{1},\ldots,N_{m}}$.
    If $N_{i}$ is a $\fancyP$-group for some $i$ then so is $N$,
    whence $N = \op{\fancyP}{N;A}$,
    contrary to $X$ not normalizing $\op{\fancyP}{N;A}$.
    We deduce that no $N_{i}$ is a $\fancyP$-group.

    \setcounter{Claim}{0}
    \begin{Claim}\label{p:6:Claim1}
        $B$ acts semiregularly on $\listset{N_{1},\ldots,N_{m}}$.
    \end{Claim}
    \begin{proof}
        Assume false.
        Set $B_{0} = \n{B}{N_{1}}$.
        Then without loss, $B_{0} \not= 1$.
        As $B \leq \zz{AX}$ and $AX$ is transitive on $\listset{N_{1},\ldots,N_{m}}$
        it follows that $B_{0}$ normalizes each $N_{i}$.
        By the same reasoning,
        either $B_{0}$ acts nontrivially on each $N_{i}$
        or trivially on each $N_{i}$.
        The minimality of $\card{A}$ rules out the latter case since
        if $[B_{0},X] = 1$
        then we could replace $A$ by $A/B_{0}$.
        Hence $B_{0}$ acts nontrivially on each $N_{i}$.
        Lemma~\ref{p:5}(b) implies that
        $\op{\fancyP}{N;A} = \op{\fancyP}{N;B_{0}}$ so as $[X,B_{0}] = 1$
        it follows from Lemma~\ref{p:4}(a) that $X$ normalizes $\op{\fancyP}{N;A}$,
        a contradiction.
    \end{proof}
    \begin{Claim}\label{p:6:Claim2}
        Let $1 \leq i \leq m$.
        Then $\card{\n{A}{N_{i}}} = r$.
    \end{Claim}
    \begin{proof}
        Without loss, $i=1$ and $\listset{N_{1},\ldots,N_{l}}$ is the orbit of
        $A$ on $\listset{N_{1},\ldots,N_{m}}$ that contains $N_{1}$.
        By Claim~\ref{p:6:Claim1},
        $\n{A}{N_{1}} \cap B = 1$ so as $B \in \hyp{A}$ it follows that
        $\card{\n{A}{N_{1}}} = 1$ or $r$.
        Suppose, for a contradiction, that $\card{\n{A}{N_{1}}} = 1$.
        Then $A$ is regular on $\listset{N_{1},\ldots,N_{l}}$.
        Set $K = N_{1} \times\cdots\times N_{l}$,
        so that $K \in \comp{A}{G}$.
        \ICiteLemmaPsix\ implies that $\cz{K}{A} \isomorphic N_{1}$
        and that $\cz{K}{A}$ is maximal subject to being $A$-invariant.
        In particular,
        $\cz{K}{A}$ is not a $\fancyP$-group so Lemma~\ref{prel:2}
        implies that $X$ normalizes $K$.
        Then $[X,\cz{K}{A}] \leq X \cap K$.
        Note that $K$ is not a $\fancyP$-group since
        $N_{1}$ is not a $\fancyP$-group,
        whence $X \cap K < K$.
        Since $X \cap K$ is $A\cz{K}{A}$-invariant,
        it follows that $X \cap K \normal \cz{K}{A}$.
        Then as $\cz{K}{A}$ is not a $\fancyP$-group and $\cz{K}{A}$ is simple,
        we have $X \cap K = 1$.
        Lemma~\ref{prel:1} implies that $[X,K] = 1$.
        Recall that $AX$ is transitive on $\listset{N_{1},\ldots,N_{m}}$
        and that $K \in \comp{A}{G}$.
        It follows that $K = N$,
        whence $[X,\op{\fancyP}{N;A}] = 1$,
        a contradiction.
    \end{proof}
    \begin{Claim}\label{p:6:Claim3}
        $\card{A} = r$ and $m = 1$, so $N$ is simple.
    \end{Claim}
    \begin{proof}
        Consider the permutation action of $AX$ on $\listset{N_{1},\ldots,N_{m}}$.
        Note that $A \in \syl{r}{AX}$ so it follows from Claim~\ref{p:6:Claim2}
        that $\n{A}{N_{i}} \in \syl{r}{\n{AX}{N_{i}}}$ for all $i$.
        Set $A^{*} = \n{A}{N_{1}}$.
        Let $1 \leq i \leq m$.
        Then $\n{A}{N_{i}}$ is conjugate in $AX$ to $A^{*}$,
        so as $A$ is abelian,
        there exists $x \in X$ with $\n{A}{N_{i}}^{x} = A^{*}$.
        Now $[\n{A}{N_{i}},x] \leq A \cap X = 1$ so we deduce that \[
            \n{A}{N_{i}} = A^{*}
        \]
        for all $i$.
        Recall that $B \in \hyp{A}$ so
        Claims~\ref{p:6:Claim1} and \ref{p:6:Claim2} imply that $A = A^{*} \times B$.
        As $X = [X,A] = \cz{X}{B}$ we have $X = [X,A^{*}]$
        and it follows that $X$ normalizes each $N_{i}$.
        Then $B$ is transitive on $\listset{N_{1},\ldots,N_{m}}$ and either
        $A^{*}$ is nontrivial on each $N_{i}$ or trivial on each $N_{i}$.
        In the latter case,
        $X$ centralizes $N$ and hence $\op{\fancyP}{N;A}$,
        a contradiction.
        Thus $A^{*}$ is nontrivial on each $N_{i}$.

        We will apply Lemma~\ref{p:5},
        with $A^{*}$ in the role of $B$.
        Put $Y = (X \cap N)\pi_{1} \times\cdots\times (X \cap N)\pi_{m}$.
        Lemma~\ref{p:5}(a) implies that $Y$ is $\cz{N}{A^{*}}$-invariant.
        Note that $XY$ is a $\fancyP$-group because $X$ normalizes each $N_{i}$
        and hence each $(X \cap N)\pi_{i}$.
        Lemma~\ref{p:5}(c) implies that
        $XY$ is an $A^{*}\cz{N}{A^{*}}$-invariant.
        Lemma~\ref{p:5}(b) implies that
        $\op{\fancyP}{N;A} = \op{\fancyP}{N;A^{*}}$
        so if $A \not= A^{*}$,
        then the minimality of $\card{A}$ supplies a contradiction.
        We deduce that $A = A^{*}$.
        Then $\card{A} = r$ and $B = 1$.
        As $B$ is transitive on $\listset{N_{1},\ldots,N_{m}}$,
        we have $m = 1$.
    \end{proof}

    It is now straightforward to complete the proof.
    Note that $X \cap N$ is an $A\cz{N}{A}$-invariant $\fancyP$-subgroup of $N$
    and that $N$ is not a $\fancyP$-group.
    Since $N$ is simple it follows that $(X \cap N)\cz{N}{A} < N$.

    Suppose that $X \cap N \leq \cz{N}{A}$.
    Then $[\cz{N}{A},X,A] \leq [X \cap N, A] = 1$.
    Trivially $[A,\cz{N}{A},X] = 1$ so the Three Subgroups Lemma forces
    $[X,A,\cz{N}{A}] = 1$.
    As $X = [X,A]$ it follows from \ICiteTheoremKfourC\
    that $[X,N] = 1$.
    Then $[X,\op{\fancyP}{N;A}] = 1$,
    a contradiction.
    We deduce that $X \cap N \not\leq \cz{N}{A}$.

    Now \ICiteTheoremKone\ implies that
    $N \isomorphic \ltwo{2^{r}}$ or $\sz{2^{r}}$
    and that $\card{\out{N}} = r$.
    Consequently \[
        G = XN = \cz{G}{N} \times N.
    \]
    Let $\alpha$ and $\beta$ be the projections
    $G \longrightarrow \cz{G}{N}$ and $G \longrightarrow N$ respectively.
    Then $X \leq X\alpha \times X\beta$ and $X\beta \leq \op{\fancyP}{N;A}$.
    It follows that $X$ normalizes $\op{\fancyP}{N;A}$,
    a contradiction.
\end{proof}

\begin{proof}[Proof of Theorem~\ref{p:2}]
    Assume false and let $G$ be a minimal counterexample.
    Using Lemma~\ref{p:4}(a) it follows that $G = \op{\fancyP}{G;A}$
    and since $\fancyP$ is closed under extensions
    we have $\op{\fancyP}{G} = 1$.
    Let $N$ be a minimal $A$-invariant normal subgroup of $G$.
    Since $G = \op{\fancyP}{G;A}$,
    the minimality of $G$ implies that $G/N$ is a $\fancyP$-group.

    Suppose that $N$ is abelian.
    Then $N$ is an elementary abelian $q$-group for some prime $q$.
    Now $\op{\fancyP}{G} = 1$ so $N$ is not a $\fancyP$-group.
    The hypothesis satisfied by $\fancyP$ implies that every $\fancyP$-group
    is a $q'$-group.
    In particular,
    $N$ is a normal Sylow subgroup of $G$.
    \ICiteCoprimeActionSZ\ implies there exists an
    $A$-invariant complement $H$ to $N$,
    so $G = HN$ and $H \cap N = 1$.
    Let $P$ be an $A\cz{G}{A}$-invariant $\fancyP$-subgroup of $G$
    and set $G_{0} = PN$.
    Then $G_{0} = (G_{0} \cap H)N$ and $P$ and $G_{0} \cap H$
    are $A$-invariant complements to $N$ in $G_{0}$.
    \ICiteCoprimeActionSZ\ implies
    $P^{c} = G_{0} \cap H$ for some $c \in \cz{G_{0}}{A}$.
    Since $P$ is $A\cz{G}{A}$-invariant we obtain $P \leq H$.
    Then $G = \op{\fancyP}{G;A} \leq H$,
    a contradiction.
    We deduce that $N$ is nonabelian.
    Then $N$ is a direct product of simple groups.

    Suppose that $N \not= G$.
    Then $\op{\fancyP}{N;A}$ is a $\fancyP$-group by the minimality of $G$.
    Now $G = \op{\fancyP}{G;A}$ so Lemma~\ref{p:6}
    implies $\op{\fancyP}{N;A} \normal G$.
    As $\op{\fancyP}{G} = 1$ this forces $\op{\fancyP}{N;A} = 1$.
    Let $P$ be an $A\cz{G}{A}$-invariant $\fancyP$-subgroup of $G$.
    Then $P \cap N \leq \op{\fancyP}{N;A} = 1$.
    Since $N$ is $A$-invariant and the direct product of simple groups,
    it is the direct product of the $A$-components of $G$.
    Lemma~\ref{prel:1} implies $[P,N] = 1$.
    But $G = \op{\fancyP}{G;A}$ whence $N \leq \zz{G}$,
    a contradiction.
    We deduce that $N = G$.
    Moreover, since $N$ is a minimal $A$-invariant normal subgroup of $G$,
    it follows that $G$ is $A$-simple.

    Recall the definitions of underdiagonal and overdiagonal subgroups of $G$
    as given in \ICiteSectionAQS.
    Let $P$ be an $A\cz{G}{A}$-invariant $\fancyP$-subgroup of $G$
    and suppose that $P$ is overdiagonal.
    Then each component of $G$ is a $\fancyP$-group,
    so $G$ is a $\fancyP$-group,
    a contradiction.
    We deduce that every $A\cz{G}{A}$-invariant $\fancyP$-subgroup
    of $G$ is underdiagonal.
    \ICiteLemmaAQSeightB\ implies that $G$ possesses a unique
    maximal $A\cz{G}{A}$-invariant underdiagonal subgroup.
    Thus $G \not= \op{\fancyP}{G;A}$.
    This final contradiction completes the proof.
\end{proof}

%% file: nap.tex
\section{Near $(A,\mathcal P)$-subgroups}\label{nap}
Throughout this section we assume the following.
\begin{Hypothesis}\label{nap:1} \mbox{}
    \begin{itemize}
        \item   $r$ is a prime and $A$ is an elementary abelian $r$-group.

        \item   $A$ acts coprimely on the $K$-group $G$.

        \item   $\fancyP$ is a group theoretic property that is closed
                under subgroups, quotients and extensions.

        \item   Every solvable group is a $\fancyP$-group.
    \end{itemize}
\end{Hypothesis}
\begin{Definition}\label{nap:2}\mbox{}
    \begin{itemize}
        \item   $G$ is a \emph{near $(A,\fancyP)$-group} if $\cz{G}{A}$ is a $\fancyP$-group.

        \item   $\op{n\fancyP}{G} = \gen{ N \normal G }%
                {\mbox{$N$ is $A$-invariant and a near $(A,\fancyP)$-group}}$.

        \item   $\op{n\fancyP}{G;A} = \gen{ H \leq G }%
                {\mbox{$H$ is $A\cz{G}{A}$-invariant and a near $(A,\fancyP)$-group}}$.
    \end{itemize}
\end{Definition}

\begin{Lemma}\label{nap:3}\mbox{}
    \begin{enumerate}
        \item[(a)]  $\op{n\fancyP}{G}$ is a near $(A,\fancyP)$-group.

        \item[(b)]  Suppose $N \normal G$ is $A$-invariant and that $N$
                    and $G/N$ are near $(A,\fancyP)$-groups.
                    Then so is $G$.
    \end{enumerate}
\end{Lemma}
\begin{proof}
    (a). Suppose that $N,M \normal G$ are $A$-invariant near $(A,\fancyP)$-groups.
    Coprime Action implies that $\cz{NM}{A} = \cz{N}{A}\cz{M}{A}$ so as
    $\cz{N}{A}$ and $\cz{M}{A}$ normalize each other,
    the assumptions on $\fancyP$ imply that $\cz{N}{A}$ and $\cz{M}{A}$
    is a $\fancyP$-group.
    Thus $NM$ is a near $(A,\fancyP)$-group and the result follows.

    (b). Because $\fancyP$ is closed under extensions.
\end{proof}

The main aim of this section is to prove the following.

\begin{Theorem}\label{nap:4}\mbox{}
    \begin{enumerate}
        \item[(a)]  $\op{n\fancyP}{G;A}$ is a near $(A,\fancyP)$-group.

        \item[(b)]  Suppose that $N$ is an $A$-invariant normal subgroup of $G$
                    then $\op{n\fancyP}{N;A} = \op{n\fancyP}{G;A} \cap N$.

        \item[(c)]  Suppose that $H$ is an $A\cz{G}{A}$-invariant subgroup of $G$ then
                    $\op{\fancyP,E}{H}$ normalizes $\op{n\fancyP}{G;A}$.
    \end{enumerate}
\end{Theorem}

\begin{Lemma}\label{nap:5}
    Suppose that $G$ is $A$-simple and that $X \not= 1$ is an
    $A\cz{G}{A}$-invariant near $(A,\fancyP)$-subgroup of $G$.
    Then $G$ is a near $(A,\fancyP)$-group.
\end{Lemma}
\begin{proof}
    Suppose first that $\sol{X} \not= 1$.
    Then $G$ possesses a nontrivial $A\cz{G}{A}$-invariant solvable subgroup.
    \ICiteTheoremKoneC\ implies that $\cz{G}{A}$ is solvable.
    By hypothesis,
    every solvable group is a $\fancyP$-group
    so $G$ is a near $(A,\fancyP)$-group.
    Hence we may assume that $\sol{X} = 1$.
    Moreover, by considering $\cz{X}{A}$ in place of $X$,
    we may assume that $\sol{\cz{X}{A}} = 1$.

    Since $\sol{X} = 1$, \ICiteTheoremKfourA\
    implies $\cz{X}{A} \not= 1$,
    so as $\sol{\cz{X}{A}} = 1$ we have
    $1 \not= \layerr{\cz{X}{A}} \normal \cz{G}{A}$.
    Now $G$ is an $A$-simple $K$-group and $\cz{G}{A}$ is nonsolvable so $\gfitt{\cz{G}{A}}$ is
    $A$-simple and $\cz{G}{A}/\gfitt{\cz{G}{A}}$ is solvable
    by \ICiteTheoremAQSfiveABandTheoremKone.
    Then $\layerr{\cz{X}{A}} = \gfitt{\cz{G}{A}}$
    so $\gfitt{\cz{G}{A}}$ is a $\fancyP$-group.
    Since $\cz{G}{A}/\gfitt{\cz{G}{A}}$ is solvable,
    the hypothesis on $\fancyP$ implies that $\cz{G}{A}$
    is a $\fancyP$-group.
    Then $G$ is a near $(A,\fancyP)$-group.
\end{proof}

\begin{Corollary}\label{nap:6}
    Suppose that $X$ is an $A\cz{G}{A}$-invariant
    near $(A,\fancyP)$-subgroup of $G$ and that $L \in \comp{A}{G}$
    with $\zz{L} = 1$.
    Then $[X,L] = 1$ or $L$ is a near $(A,\fancyP)$-group.
\end{Corollary}
\begin{proof}
    Assume $[X,L] \not= 1$.
    Lemma~\ref{prel:1} implies that $X \cap L \not= 1$ so as $X \cap L$
    is an $A\cz{L}{A}$-invariant near $(A,\fancyP)$-subgroup of $L$,
    the result follows.
\end{proof}

\begin{proof}[Proof of Theorem~\ref{nap:4}]
    (a). Assume false and let $G$ be a minimal counterexample.
    Then $G = \op{n\fancyP}{G;A}, \op{n\fancyP}{G} = 1, \sol{G} = 1$
    and $G/\layerr{G}$ is a near $(A,\fancyP)$-group.
    It follows that $\layerr{G}$ is not a near $(A,\fancyP)$-group
    and that there exists $L \in \comp{A}{G}$ such that $L$ is
    not a near $(A,\fancyP)$-group.
    Note that $\zz{L} \leq \sol{G} = 1$.
    As $G = \op{n\fancyP}{G;A}$,
    Corollary~\ref{nap:6} implies $L \leq \zz{G}$,
    a contradiction.

    (b). Since $\op{n\fancyP}{G;A} \cap N$ is a near $(A,\fancyP)$-group
    we have $\op{n\fancyP}{G;A} \cap N \leq \op{\fancyP}{N;A}$.
    Now $\cz{G}{A}$ permutes the $A\cz{N}{A}$-invariant
    near $(A,\fancyP)$-subgroups of $N$ so $\cz{G}{A}$
    normalizes $\op{\fancyP}{N;A}$.
    Thus $\op{n\fancyP}{N;A} \leq \op{\fancyP}{G;A} \cap N$,
    completing the proof.

    (c). Since $\op{n\fancyP}{G} \leq \op{n\fancyP}{G;A}$
    we may pass to the quotient $G/\op{n\fancyP}{G}$ and
    assume that $\op{n\fancyP}{G} = 1$.
    Now every solvable group is a $\fancyP$-group so
    $\sol{G} = 1$ and hence $\cz{G}{\layerr{G}} = 1$.
    Let
    \begin{align*}
        \mathcal C_{n} &= \set{ K \in \comp{A}{G} }%
                {\mbox{$K$ is a near $(A,\fancyP)$-group}} \\
        \mathcal C_{0} &= \comp{A}{G} - \mathcal C_{n}.
    \end{align*}
    Corollary~\ref{nap:6} implies that $\op{n\fancyP}{G;A}$
    centralizes $\listgen{\mathcal C_{0}}$.
    Then $\op{n\fancyP}{G;A}$ normalizes
    $\listgen{\mathcal C_{n}} = \cz{\layerr{G}}{\listgen{\mathcal C_{0}}}$.
    As $\cz{G}{\layerr{G}} = 1$ we also have $\mathcal C_{n} \not= \emptyset$.

    Suppose $L \in \comp{A,\fancyP}{H}$.
    We claim that $L$ acts trivially on $\mathcal C_{0}$.
    If $L \leq \layerr{G}$ then the claim is trivial
    so suppose $L \not\leq \layerr{G}$.
    Now every solvable group is a $\fancyP$-group so $\op{\fancyP}{L}$
    is the unique maximal $A$-invariant normal subgroup of $L$.
    Consequently $L \cap \layerr{G} \leq \op{\fancyP}{L}$.
    Now $\layerr{G} \cap H\cz{G}{A} \normal H\cz{G}{A}$ so
    Theorem~\ref{ap:4}(d) implies that $\layerr{G} \cap H\cz{G}{A}$
    normalizes $L$.
    Let $K \in \mathcal C_{0}$.
    Then $\cz{K}{A} \leq \n{G}{L}$.
    Since $K \cap L \leq L \cap \layerr{G} \leq \op{\fancyP}{L}$,
    Lemma~\ref{nap:5} implies $K \cap L = 1$,
    whence $[K,L] = 1$ by Lemma~\ref{prel:1}.
    This establishes the claim.
    We deduce that $\layer{\fancyP}{H}$ normalizes $\listgen{\mathcal C_{0}}$
    and also normalizes
    $\listgen{\mathcal C_{n}} = \cz{\layerr{G}}{\listgen{\mathcal C_{0}}}$.
    We have previously seen that $\op{n\fancyP}{G;A}$ normalizes
    $\listgen{\mathcal C_{n}}$ so as $\op{\fancyP}{H} \leq \op{n\fancyP}{G;A}$
    we have that \[
        \listgen{ \op{\fancyP,E}{H}, \op{n\fancyP}{G;A}, \cz{G}{A} } %
        \leq \n{G}{\listgen{\mathcal C_{n}}}.
    \]
    Since $\op{n\fancyP}{G} = 1$,
    the normalizer is a proper subgroup of $G$.
    The conclusion follows by induction.
\end{proof}

%% file: lglob.tex
\section{Local to global results}\label{lglob}
We shall generalize some of the results of \ICiteSectionL\
concerning $A$-components to $(A,\fancyP)$-components.
Consider the following:
\begin{Hypothesis}\label{lglob:1}\mbox{}
    \begin{itemize}
        \item   $r$ is a prime and $A$ is an elementary abelian
                $r$-group that acts coprimely on the $K$-group $G$.

        \item   $\fancyP$ is a group theoretic property that
                satisfies Hypothesis~\ref{pc:1}.

        \item   $H$ is an $A\cz{G}{A}$-invariant subgroup of $G$.
    \end{itemize}
\end{Hypothesis}

The aim is to establish a connection between the $(A,\fancyP)$-components
of $H$ and the structure of $G$.
This is not possible in full generality,
but if additional assumptions are made then it is.

\begin{Hypothesis}\label{lglob:2}\mbox{}
    \begin{itemize}
        \item   Hypothesis~\ref{lglob:1}.

        \item   Every $\fancyP$-group is solvable.

        \item   $a \in A\nonid$ and $H$ is $\cz{G}{a}$-invariant.
    \end{itemize}
\end{Hypothesis}

\begin{Hypothesis}\label{lglob:3}\mbox{}
    \begin{itemize}
        \item   Hypothesis~\ref{lglob:1}.

        \item   Whenever $A$ acts coprimely on the $K$-group $X$
                and $\cz{X}{A}$ is a $\fancyP$-group then $X$ is solvable.
    \end{itemize}
\end{Hypothesis}

\begin{Lemma}\label{lglob:4}\mbox{}
    \begin{enumerate}
        \item[(a)]  Assume Hypothesis~\ref{lglob:1}.
                    If $\fancyP$ is any of the properties
                    ``is trivial'', ``is nilpotent'' or ``has odd order''
                    then Hypothesis~\ref{lglob:3} is satisfied.

        \item[(b)]  Assume Hypothesis~\ref{lglob:3}.
                    Then every $\fancyP$-group is solvable.
    \end{enumerate}
\end{Lemma}
\begin{proof}
    (a). This follows from \ICiteTheoremKfour.

    (b). Let $X$ be a $\fancyP$-group and let $A$ act trivially on $X$.
\end{proof}

\noindent We state the main result of this section.
\begin{Theorem}\label{lglob:5}\mbox{}
    \begin{enumerate}
        \item[(a)]  Assume Hypothesis~\ref{lglob:2} and that $K \in \comp{A,\fancyP}{H}$
                    satisfies $K = [K,a]$.
                    Then $K \in \comp{A,\fancyP}{G}$.

        \item[(b)]  Assume Hypothesis~\ref{lglob:3}.
                    Then $\op{\fancyP}{H}\layer{\fancyP}{H}$
                    acts trivially on $\compsol{G}$.
                    If $K \in \comp{A,\fancyP}{H}$ then there exists a
                    unique $\widetilde{K} \in \compasol{G}$
                    with $K \leq \widetilde{K}$.
    \end{enumerate}
\end{Theorem}

\noindent At the end of this section,
examples will be constructed to show that the additional hypothesis in (b)
is needed.
It would be interesting to investigate if (a) holds without the solvability hypothesis.
Two lemmas are required for the proof of Theorem~\ref{lglob:5}.

\begin{Lemma}\label{lglob:6}
    Assume Hypothesis~\ref{lglob:1},
    that $K \in \comp{A,\fancyP}{H}$ and that
    \begin{enumerate}
        \item[(a)]  $K \leq \layerr{G}$, or
        \item[(b)]  $\sol{G} = 1$ and $K$ acts trivially on $\compp{G}$.
    \end{enumerate}
    Then there exists a unique $\widetilde{K} \in \comp{A}{G}$
    with $K \leq \widetilde{K}$.
\end{Lemma}
\begin{proof}
    Uniqueness is clear since distinct elements of $\comp{A}{G}$
    have solvable intersection.
    Note that since $H$ is $\cz{G}{A}$-invariant we have
    $K \in \comp{A,\fancyP}{H\cz{G}{A}}$.
    Hence we may assume that $\cz{G}{A} \leq H$.

    Suppose (a) holds.
    Assume the conclusion to be false.
    Let $L \in \comp{A}{G}$.
    Now $L \cap H \subnormal H$ and $K \not\leq L$ so
    Lemma~\ref{ap:4}(d) implies $[K,L \cap H] \leq \op{\fancyP}{K}$.
    As $\cz{L}{A} \leq H$ we have $[K,\cz{L}{A}] \leq \op{\fancyP}{K}$
    and $\cz{L}{A}$ normalizes $K$.
    Since $\cz{\layerr{G}}{A}$ is the product of the subgroups $\cz{L}{A}$
    as $L$ ranges over $\comp{A}{G}$,
    it follows that $[K,\cz{\layerr{G}}{A}] \leq \op{\fancyP}{K}$.

    By hypothesis, $K \leq \layerr{G}$ so $[K,\cz{K}{A}] \leq \op{\fancyP}{K}$.
    Hence $\op{\fancyP}{K}\cz{K}{A} \normal K$ and Coprime Action
    implies that $A$ is fixed point free on $K/\op{\fancyP}{K}\cz{K}{A}$.
    This quotient is therefore solvable by \ICiteTheoremKfour.
    Since $K$ is perfect,
    it follows that $K = \op{\fancyP}{K}\cz{K}{A}$.
    But $[K,\cz{K}{A}] \leq \op{\fancyP}{K}$ so $K/\op{\fancyP}{K}$
    is abelian.
    This contradiction completes the proof.

    Suppose (b) holds.
    Let $N$ be the intersection of the normalizers of the components of $G$.
    Since $\sol{G} = 1$ we have $\cz{G}{\layerr{G}} = 1$ and then the
    Schreier Property implies that $N/\layerr{G}$ is solvable.
    Now $K$ is perfect and $K \leq N$ whence $K \leq \layerr{G}$.
    Apply (a).
\end{proof}

\begin{Lemma}\label{lglob:7}
    Assume Hypothesis~\ref{lglob:1} and that $L \in \comp{A}{G}$.
    Then:
    \begin{enumerate}
        \item[(a)]  $\op{\fancyP}{H}\layer{\fancyP}{H}$ normalizes $L$
                    and every component of $L$; or

        \item[(b)]  $\cz{L/\zz{L}}{A}$ is a $\fancyP$-group.
    \end{enumerate}
\end{Lemma}
\begin{proof}
    Suppose that $\op{\fancyP}{H}\layer{\fancyP}{H} \leq \n{G}{L}$.
    Then $A\op{\fancyP}{H}\layer{\fancyP}{H}$ permutes $\compp{L}$.
    Since $A$ is transitive and $\op{\fancyP}{H}\layer{\fancyP}{H}$ is a normal Hall-subgroup
    of $A\op{\fancyP}{H}\layer{\fancyP}{H}$,
    \ICiteLemmaPeight\ implies that $\op{\fancyP}{H}\layer{\fancyP}{H}$
    acts trivially.
    Then (a) holds.

    Suppose that $\op{\fancyP}{H} \not\leq \n{G}{L}$.
    Now $\cz{L}{A} \leq \n{G}{H}$ so $[\op{\fancyP}{H},\cz{L}{A}] \leq \op{\fancyP}{H}$.
    In particular, the commutator is a $\fancyP$-group.
    Lemma~\ref{prel:2} implies that $\cz{L/\zz{L}}{A}$ is a $\fancyP$-group.

    Suppose that $\layer{\fancyP}{H} \not\leq \n{G}{L}$.
    Choose $K \in \comp{A,\fancyP}{H}$ with $K \not\leq \n{G}{L}$.
    Set $H_{0} = H\cz{G}{A}$ so $H \normal H_{0}$ and $K \in \comp{A,\fancyP}{H_{0}}$.
    Now $L \cap H_{0} \subnormal H_{0}$ so Lemma~\ref{pc:4}(d) implies $K \leq L \cap H_{0}$
    or $[K,L\cap H_{0}] \leq \op{\fancyP}{K}$.
    The first possibility does not hold since $K \not\leq \n{G}{L}$.
    Moreover, $\cz{L}{A} \leq L \cap H_{0}$ so we deduce that $[K,\cz{L}{A}]$ is a $\fancyP$-group.
    Lemma~\ref{prel:2},
    with $K$ and $L$ in the roles of $P$ and $K$ respectively,
    implies that $\cz{L/\zz{L}}{A}$ is a $\fancyP$-group.
\end{proof}

\begin{proof}[Proof of Theorem~\ref{lglob:5}]
    (a). Now $K \in \comp{A,\fancyP}{\cz{G}{a}H}$ because $\cz{G}{a}$ normalizes $H$.
    Hence we may assume that $\cz{G}{a} \leq H$.
    Now $K = K_{1} * \cdots * K_{n}$ where $K_{1},\ldots,K_{n}$ are the
    $\listgen{a}$-components of $K$.
    As $K = [K,a]$ it follows that $K_{i} = [K_{i},a]$ for each $i$.
    If $K_{i} \subnormal G$ for each $i$ then $K \subnormal G$ and then
    $K \in \comp{A,\fancyP}{G}$.
    Hence, as Hypothesis~\ref{lglob:2} remains valid if $A$ is replaced by
    $\listgen{a}$, we may assume that $A = \listgen{a}$.

    Consider first the case that $\sol{G} = 1$.
    Assume that $K$ acts nontrivially on $\compp{G}$.
    Then the set \[
        \mathcal C = \set{L_{0} \in \compp{G}}{ K \not\leq \n{G}{L_{0}} }
    \]
    is nonempty.
    Since $A$ normalizes $K$ it follows that $A$ acts on $\mathcal C$.
    Now $K = [K,a]$ so it follows also that $a$ acts nontrivially.
    Choose $L_{0} \in \mathcal C$ with $L_{0} \not= L_{0}^{a}$.
    Set $L = \listgen{L_{0}^{A}} \in \comp{A}{G}$.
    Now $L_{0}$ is simple because $\sol{G} = 1$ so as $A = \listgen{a}$
    we see that $L$ is the direct product of $r$ copies of $L_{0}$
    and then that $\cz{L}{A}  \isomorphic L_{0}$.
    By hypothesis, every $\fancyP$-group is solvable,
    so $\cz{L}{A}$ is not \ $\fancyP$-group.
    Lemma~\ref{lglob:7} implies that $K$ normalizes $L$ and $L_{0}$,
    a contradiction.
    We deduce that $K$ acts trivially on $\compp{G}$.

    Lemma~\ref{lglob:6} implies there exists $\widetilde{K}$ with
    $K \leq \widetilde{K} \in \comp{A}{G}$.
    Now $\cz{\widetilde{K}}{a} \leq H \cap \widetilde{K}$
    and $K = [K,a] \leq H \cap \widetilde{K}$
    so $\cz{\widetilde{K}}{a} < H \cap \widetilde{K}$.
    \ICiteCorollaryAQSnine\ implies $H \cap \widetilde{K} = \widetilde{K}$.
    Then $K \subnormal \widetilde{K}$.
    Since $\widetilde{K}$ is $A$-simple we obtain $K = \widetilde{K}$
    and so $K \subnormal G$ as desired.

    Returning now to the general case,
    set $S = \sol{G}$.
    Applying the previous argument to $G/S$,
    we obtain $KS \subnormal G$.
    Now $\cz{S}{a} \leq S \cap H \leq \sol{H}$ so Lemma~\ref{pc:4}(i)
    implies $\cz{S}{a} \leq \n{G}{K}$.
    Then \ICiteLemmaGthree\ implies $S \leq \n{G}{K}$
    whence $K \subnormal G$ and $K \in \comp{A,\fancyP}{G}$ as required.

    (b). As in (a) we may suppose that $\cz{G}{A} \leq H$.
    Consider first the case that $\sol{G} = 1$.
    If $L \in \comp{A}{G}$ then $L$ is nonsolvable so by hypothesis,
    $\cz{L}{A}$ is not a $\fancyP$-group.
    Lemma~\ref{lglob:7} implies that $\op{\fancyP}{H}\layer{\fancyP}{H}$
    normalizes $L$ and every component of $L$.
    Since every component of $G$ is contained in an $A$-component of $G$,
    it follows that $\op{\fancyP}{H}\layer{\fancyP}{H}$ acts trivially on $\compp{G}$.
    Lemma~\ref{lglob:6}(b) implies there exists $\widetilde{K}$
    with $K \leq \widetilde{K} \in \comp{A}{G}$.

    Returning to the general case,
    set $S = \sol{G}$ and $\br{G} = G/S$.
    The map $X \mapsto \br{X}$ is a bijection $\compsol{G} \longrightarrow \compp{G}$.
    It follows that $\op{\fancyP}{H}\layer{\fancyP}{H}$ acts trivially on $\compsol{G}$.
    By the previous paragraph and Lemma~\ref{ap:4}(j) there exists $M \in \compasol{G}$
    with $\br{K} \leq \br{M}$.
    Then $K \leq M\sol{G}$.
    Lemma~\ref{ap:4}(i) implies $M \normal M\sol{G}$ whence
    $K = K^{\infty} \leq (M\sol{G})^{\infty} = M$ and the proof is complete.
\end{proof}

We close this section with a corollary and an example.
In what follows, \emph{nil} is an abbreviation for the group theoretic
property ``is nilpotent''.

\begin{Corollary}\label{lglob:8}
    Let $A$ be an elementary abelian $r$-group that acts coprimely
    on the $K$-group $G$.
    Let $a \in A\nonid$ and suppose that $H$ is an $A\cz{G}{a}$-invariant
    subgroup of $G$.
    \begin{enumerate}
        \item[(a)]  Let $K \in \comp{A}{H}$.
                    Then there exists $\widetilde{K}$ with \[
                        K \leq \widetilde{K} \in \comp{A,\mathrm{nil}}{G}.
                    \]

        \item[(b)]  Let $K \in \comp{A,\mathrm{nil}}{H}$.
                    Then there exists $\widetilde{K}$ with \[
                        K \leq \widetilde{K} \in \compasol{G}.
                    \]
    \end{enumerate}
\end{Corollary}
\begin{proof}
    (a) follows from \ICiteTheoremLone\ and (b) follows from
    Lemma~\ref{lglob:4} and Theorem~\ref{lglob:5}(a).
\end{proof}

\noindent The following example shows that the corollary cannot be extended
further and that the restriction on $\fancyP$ in Theorem~\ref{lglob:5}(b)
is needed.

\begin{Example}
    Let $r$ be a prime and $J_{1}$ be a simple $r'$-group that admits
    an automorphism $a_{1}$ of order $r$ with $\cz{J_{1}}{a_{1}}$ solvable.
    For example $\ltwo{2^{r}}$ or $\sz{2^{r}}$ with $r > 5$.
    Let $K$ be a simple group with order $n$.
    Then $K$ acts on the direct product \[
        D = J_{1}\listgen{a_{1}} \times J_{2}\listgen{a_{2}} %
        \times\cdots\times J_{n}\listgen{a_{n}}
    \]
    permuting the direct factors regularly.
    Set $a = a_{1}a_{2}\cdots a_{n} \in \cz{D}{K}$
    and put $A = \listgen{a}$.
    Let $G$ be the semidirect product \[
        G = (J_{1} \times\cdots\times J_{n}) \rtimes K.
    \]
    Then $a$ acts on $G$ with
    $\cz{G}{a} = (\cz{J_{1}}{a_{1}} \times\cdots\times \cz{J_{n}}{a_{n}})K$.
    Observe that $K$ is contained in an $(A,\mbox{sol})$-component of $\cz{G}{a}$.
    However, the $(A,\mbox{sol})$-components of $G$ are $J_{1},\ldots,J_{n}$,
    none of which contain $K$.
\end{Example}

%% file: md.tex
\section{The McBride Dichotomy}\label{md}
In this section we give an application of \S\ref{nap}
to the study of Signalizer Functors.
No originality is claimed,
the results being a presentation of McBride's work \cite[Theorem~6.6]{McB1}.
They culminate in a fundamental dichotomy in the
proof of the Nonsolvable Signalizer Functor Theorem.

Throughout this section, we assume the following.
\begin{Hypothesis}\label{md:1}\mbox{}
    \begin{enumerate}
        \item[(a)]  $\fancyP$ is a group theoretic property that
                    is closed under subgroups, quotients and extensions.

        \item[(b)]  Every solvable group is a $\fancyP$-group.
    \end{enumerate}
\end{Hypothesis}
\noindent We will be interested in the subgroup $\op{\fancyP,E}{H}$
for groups $H$.
Note that \[
    \op{\fancyP,E}{H} = \op{\fancyP}{H}\layer{\fancyP}{H}.
\]
\begin{Lemma}\label{md:2}
    Let $G$ be a group.
    \begin{enumerate}
        \item[(a)]  Suppose $H \leq G$ and $\op{\fancyP,E}{G} \leq H$.
                    Then $\op{\fancyP,E}{G}= \op{\fancyP,E}{H}$.

        \item[(b)]  Suppose $H,M \leq G$, $\op{\fancyP,E}{H} \leq M$
                    and $\op{\fancyP,E}{M} \leq H$.
                    Then $\op{\fancyP,E}{H} = \op{\fancyP,E}{M}$.
    \end{enumerate}
\end{Lemma}
\begin{proof}
    (a). Set $\br{G} = G/\op{\fancyP}{G}$.
    Then \[
        \layerr{\br{G}} = \br{\op{\fancyP,E}{G}} \leq \br{H}.
    \]
    Since $\fancyP$ is closed under extensions we have $\op{\fancyP}{\br{G}} = 1$.
    Then
    \begin{align*}
        [\op{\fancyP}{\br{H}},\layerr{\br{G}}] &\leq %
            \op{\fancyP}{\br{H}} \cap \layerr{\br{G}} \\
            &\leq \op{\fancyP}{\layerr{\br{G}}} \leq \op{\fancyP}{\br{G}} = 1.
    \end{align*}
    Now every solvable group is a $\fancyP$-group so $\sol{\br{G}} = 1$
    and hence $\cz{\br{G}}{\layerr{\br{G}}} = 1$.
    Then $\op{\fancyP}{\br{H}} = 1$ and $\op{\fancyP}{H} \leq \op{\fancyP}{G}$.
    Since $\op{\fancyP}{G} \leq H$
    it follows that $\op{\fancyP}{H} = \op{\fancyP}{G}$.
    As $\layerr{\br{G}} \leq \br{H}$ we have
    $\layerr{\br{G}} \normal \layerr{\br{H}}$.
    Then any component of $\br{H}$ not contained in $\layerr{\br{G}}$
    would centralize $\layerr{\br{G}}$,
    contrary to $\cz{\br{G}}{\layerr{\br{G}}} = 1$.
    We deduce that $\layerr{\br{G}} = \layerr{\br{H}}$
    and the conclusion follows.

    (b). Observe that
    $\op{\fancyP,E}{H} = \op{\fancyP,E}{H \cap M} = \op{\fancyP,E}{M}$.
\end{proof}
\noindent We remark that (b) is an elementary version of Bender's Maximal Subgroup Theorem.

\begin{Lemma}\label{md:3}
    Let $r$ be a prime and $A$ be a noncyclic elementary abelian $r$-group
    that acts coprimely on the $K$-group $G$.
    Then
    \begin{align*}
        \op{\fancyP,E}{G} &=    \op{\fancyP,E}{ \gen{\op{\fancyP,E}{\cz{G}{B}}}{ B \in \hyp{A} } }\\
                        &=    \op{\fancyP,E}{ \gen{\op{\fancyP,E}{\cz{G}{a}}}{ a \in A\nonid } }.
    \end{align*}
\end{Lemma}
\begin{proof}
    Let $H$ be either right hand side.
    Coprime Action implies that $\op{\fancyP}{G} \leq H$.
    Passing to the quotient $G/\op{\fancyP}{G}$ and applying
    \ICiteLemmaAQStwelve\ it follows that $\op{\fancyP,E}{G} \leq H$.
    Apply Lemma~\ref{md:2}(a).
\end{proof}

Throughout the remainder of this section we assume the following.
\begin{Hypothesis}\label{md:4}\mbox{}
    \begin{enumerate}
        \item[(a)]  Hypothesis~\ref{md:1}.
        \item[(b)]  $r$ is a prime and $A$ is an elementary abelian $r$-group
                    with rank at least $3$.
        \item[(c)]  $A$ acts on the group $G$.
        \item[(d)]  $\theta$ is an $A$-signalizer functor on $G$.
        \item[(e)]  $\theta(a)$ is a $K$-group for all $a \in A\nonid$.
        \item[(f)]  $\widetilde{G} = \gen{ \theta(a) }{ a \in A\nonid }$.
    \end{enumerate}
\end{Hypothesis}

\begin{Lemma}\label{md:5}
    $\widetilde{G} = \gen{ \theta(B) }{ B \in \hyp{A} }.$
\end{Lemma}
\begin{proof}
    Let $a \in A\nonid$.
    Coprime Action applied to the action of $A/\listgen{a}$
    on $\theta(a)$ implies that \[
        \theta(a) = \gen{ \theta(a) \cap \cz{G}{B} }{ a \in B \in \hyp{A} }.
    \]
    Now $\theta(a) \cap \cz{G}{B} = \theta(B)$ whenever $a \in B \in \hyp{A}$
    so the conclusion follows.
\end{proof}
\begin{Lemma}\label{md:6}
    Let
    \begin{align*}
        H_{1}   &=  \gen{\op{\fancyP,E}{\theta(a)}}{a \in A\nonid} \\
    \intertext{and}
        H_{2}   &=  \gen{\op{\fancyP,E}{\theta(B)}}{B \in \hyp{A}}.
    \end{align*}
    Let $i \in \listset{1,2}$ and suppose that $H_{i}$ is contained in
    a $\theta$-subgroup.
    Then $H_{i}$ is a $\theta$-subgroup and
    $\op{\fancyP,E}{H_{1}} \normal \widetilde{G}$.
\end{Lemma}
\begin{proof}
    Since any $A$-invariant subgroup of a $\theta$-subgroup
    is a $\theta$-subgroup,
    the first assertion holds.
    Suppose $i=1$.
    Let $B \in \hyp{A}$.
    Lemma~\ref{md:3},
    with $B$ in the role of $A$,
    implies that \[
        \op{\fancyP,E}{H_{1}} = %
        \op{\fancyP,E}{\gen{\op{\fancyP,E}{\cz{H_{1}}{b}}}{b \in B\nonid}}.
    \]
    Let $b \in B\nonid$.
    Since $H_{1}$ is a $\theta$-subgroup we have
    $\op{\fancyP,E}{\theta(b)} \leq \cz{H_{1}}{b} \leq \theta(b)$.
    Lemma~\ref{md:2}(a) implies that \[
        \op{\fancyP,E}{\theta(b)} = \op{\fancyP,E}{\cz{H_{1}}{b}}.
    \]
    Note that $\theta(B) \leq \theta(b)$ so $\theta(B)$
    normalizes $\op{\fancyP,E}{\theta(b)}$.
    It follows that $\theta(B)$ normalizes $\op{\fancyP,E}{H_{1}}$.
    Then Lemma~\ref{md:5} implies that $\op{\fancyP,E}{H_{1}} \normal \widetilde{G}$.

    Suppose $i=2$.
    Let $a \in A\nonid$.
    Lemma~\ref{md:3} implies that
    \begin{align*}
        \op{\fancyP,E}{\theta(a)}   &= %
            \op{\fancyP,E}{\gen{\op{\fancyP,E}{\theta(a) \cap \cz{G}{B}}}{a \in B \in \hyp{A}}} \\
                                    &= %
            \op{\fancyP,E}{\gen{\op{\fancyP,E}{\theta(B)}}{B \in \hyp{A}}} \\
                                    &\leq H_{2}.
    \end{align*}
    It follows that $H_{1} \leq H_{2}$,
    so apply the previous case.
\end{proof}

For each $a \in A\nonid$,
define \[
    \theta_{n\fancyP}(a) = \op{n\fancyP}{\theta(a);A}.
\]
Theorem~\ref{nap:4} implies that $\theta_{n\fancyP}$
is an $A$-signalizer functor on $G$.

\begin{Lemma}\label{md:7}
    Assume that $\theta_{n\fancyP}$ is complete.
    Then \[
        \op{\fancyP,E}{\theta(B)} \leq \n{G}{\theta_{n\fancyP}(G)}
    \]
    for all $B \in \hyp{A}$.
\end{Lemma}
\begin{proof}
    Set $S = \theta_{n\fancyP}(G)$ and let $b \in B\nonid$.
    Then \[
        \cz{S}{b} = \theta_{n\fancyP}(b) = \op{n\fancyP}{\theta(b);A}.
    \]
    Now $\theta(B) \leq \theta(b)$ so Theorem~\ref{nap:4}(c) implies
    $\op{\fancyP,E}{\theta(B)}$ normalizes $\op{n\fancyP}{\theta(b);A}$.
    Since $S = \gen{\cz{S}{b}}{b \in B\nonid}$,
    the result follows.
\end{proof}

\begin{Theorem}[The McBride Dichotomy]\label{md:8}
    Suppose that $\theta$ is a minimal counterexample to the
    Nonsolvable Signalizer Functor Theorem.
    \begin{enumerate}
        \item[(a)]  Either $\theta_{n\fancyP} = 1$ or $\theta = \theta_{n\fancyP}$.

        \item[(b)]  Either
                    \begin{itemize}
                        \item   There exist no nontrivial $\theta(A)$-invariant
                                solvable $\theta$-subgroups; or
                        \item    $\theta(A)$ is solvable and every nonsolvable
                                composition factor of every $\theta$-subgroup
                                belongs to $\badfour$.
                    \end{itemize}
    \end{enumerate}
\end{Theorem}
\begin{proof}
    Note that minimality is with reference to the integer \[
        \card{\theta} = \sum_{a \in A\nonid} \card{\theta(a)}.
    \]
    Since $\theta$ is a minimal counterexample,
    it follows that $G = \widetilde{G}$ and that no nontrivial
    $\theta$-subgroup is normal in $G$.

    (a). Suppose that $\theta_{n\fancyP}$ is complete
    and that $\theta_{n\fancyP} \not= 1$.
    Set $S = \theta_{n\fancyP}(G)$.
    Then $S$ is a $\theta$-subgroup so $\n{G}{S} \not= G$
    and hence $\n{G}{S}$ possesses a unique maximal $\theta$-subgroup,
    $\theta(\n{G}{S})$.
    Lemmas~\ref{md:7} and \ref{md:6} supply a contradiction.
    We deduce that either $\theta_{n\fancyP}$ is not complete,
    in which case $\theta = \theta_{n\fancyP}$;
    or $\theta_{n\fancyP}(G) = 1$,
    in which case $\theta_{n\fancyP} = 1$.

    (b). Let $\fancyP$ be the group theoretic property ``is solvable''.
    Suppose that $\theta_{n\fancyP} = 1$.
    Let $X$ be a $\theta(A)$-invariant solvable $\theta$-subgroup.
    Let $a \in A\nonid$.
    Then, as $\cz{X}{A}$ is solvable,
    we have
    $\cz{X}{a} \leq \op{n\fancyP}{\theta(a);A} = \theta_{n\fancyP}(a) = 1$.
    Since $A$ is noncyclic,
    it follows that $X = 1$,
    so the first assertion holds.
    Suppose that $\theta = \theta_{n\fancyP}$.
    Let $a \in A\nonid$.
    Then $\theta(A) = \cz{\theta(a)}{A} = \cz{\theta_{n\fancyP}(a)}{A}$
    so $\theta(A)$ is solvable.
    Let $X$ be a $\theta$-subgroup.
    Then $\cz{X}{A} \leq \theta(A)$ so $\cz{X}{A}$ is solvable.
    The conclusion follows from \ICiteTheoremKfour.
\end{proof}

%% file: gor.tex
\section{A theorem of Gorenstein and Lyons}\label{gor}
We will provide an alternate proof of a special case
of the Nonsolvable Signalizer Functor Theorem due
to Gorenstein and Lyons \cite{GL}.
It is an application of the main result of \S\ref{p}.
Throughout this section,
we assume the following.
\begin{Hypothesis}\label{gor:1}\mbox{}
    \begin{itemize}
        \item   $r$ is a prime and $A$ is an elementary abelian $r$-group
                with rank at least $3$.
        \item   $A$ acts on the group $G$.
        \item   $\theta$ is an $A$-signalizer functor on $G$.
        \item   $\theta(a)$ is a $K$-group for all $a \in A\nonid$.
    \end{itemize}
\end{Hypothesis}
\noindent We shall prove:
\begin{Theorem}[Gorenstein-Lyons]\label{gor:2}
    Assume that $A$ acts trivially on $\compsol{\theta(a)}$
    for all $a \in A\nonid$.
    Then $\theta$ is complete.
\end{Theorem}

\noindent First we develop a little general theory.

\begin{Definition}\label{gor:3}
    A \emph{subfunctor} of $\theta$ is an $A$-signalizer functor
    $\psi$ on $G$ with $\psi(a) \leq \theta(a)$ for all $a \in A\nonid$.
    We say that $\psi$ is a \emph{proper subfunctor} if $\psi(a) \not= \theta(a)$
    for some $a \in A\nonid$ and that $\psi$ is \emph{$\theta(A)$-invariant}
    if $\psi(a)$ is normalized by $\theta(A)$ for all $a \in A\nonid$.
\end{Definition}
\begin{Lemma}\label{gor:4}
    Let $t \in A\nonid$,
    set $D = \gen{ \cz{[\theta(a),t]}{t} }{ a \in A\nonid }$
    and define $\psi$ by \[
        \psi(a) = [\theta(a),t](\theta(a) \cap D)
    \]
    for all $a \in A\nonid$.
    \begin{enumerate}
        \item[(a)]  $\psi$ is a $\theta(A)$-invariant subfunctor of $\theta$.
        \item[(b)]  If $\psi$ is complete then so is $\theta$.
    \end{enumerate}
\end{Lemma}
\begin{proof}
    (a). Let $a \in A\nonid$.
    Then $\cz{[\theta(a),t]}{t}$ is $A\theta(A)$-invariant
    since $\theta(a)$ and $t$ are.
    Thus $D$ is $A\theta(A)$-invariant.
    Now $[\theta(a),t] \normal \theta(a)$ so $\psi(a)$ is a subgroup of $\theta(a)$.
    Again, it is $A\theta(A)$-invariant.
    Let $b \in A\nonid$.
    By \ICiteCoprimeActionProd,
    \[
        \psi(a) \cap \cz{G}{b} = \cz{[\theta(a),t]}{b}\cz{\theta(a) \cap D}{b}.
    \]
    Set $X = \cz{[\theta(a),t]}{b}$.
    Then $X \leq \theta(a) \cap \cz{G}{b} \leq \theta(b)$.
    By \ICiteCoprimeActionComm\ we have \[
        X = [X,t]\cz{X}{t} \leq %
        \listgen{ [\theta(b),t], \theta(b) \cap \cz{[\theta(a),t]}{t} } %
            \leq \psi(b).
    \]
    Trivially, $\cz{\theta(a) \cap D}{b} \leq \theta(b) \cap D \leq \psi(b)$.
    We conclude that $\psi(a) \cap \cz{G}{b} \leq \psi(b)$,
    so $\psi$ is an $A$-signalizer functor.

    (b). This is \cite[Corollary~4.3]{F2}
\end{proof}

\begin{Lemma}\label{gor:5}
    Suppose that:
    \begin{enumerate}
        \item[(i)]  $\theta$ is incomplete.
        \item[(ii)]  $\psi$ is complete whenever $\psi$ is a
                    proper $\theta(A)$-invariant subfunctor of $\theta$.
    \end{enumerate}
    Then the following hold:
    \begin{enumerate}
        \item[(a)]  For each $t \in A\nonid$, \[
                        \theta(t) = \gen{ \cz{[\theta(a),t]}{t} }{ a \in A\nonid }.
                    \]

        \item[(b)]  Let \[
            \mathcal S = \set{ S }{ \mbox{$S$ is a simple section of %
                            $\cz{[\theta(a),t]}{t}$ for some $a,t \in A\nonid$ }}
                    \]
                    and let $\fancyP$ be the group theoretic property defined by:
                    \begin{quote}
                        $H$ is a $\fancyP$-group if and only if every noncyclic
                        composition factor of $H$ is isomorphic to a member of $\mathcal S$.
                    \end{quote}
                    Then $\theta(t)$ is a $\fancyP$-group for all $t \in A\nonid$.
    \end{enumerate}
\end{Lemma}
\begin{proof}
    (a). Adopt the notation defined in the statement of Lemma~\ref{gor:5}.
    Since $\theta$ is incomplete,
    it follows from (ii) and Lemma~\ref{gor:4} that $\theta = \psi$.
    Then $\theta(t) = \psi(t) = [\theta(t),t](\theta(t) \cap D) = D$.

    (b). For each $a \in A\nonid$,
    $\cz{[\theta(a),t]}{t}$ is an $A\theta(A)$-invariant $\fancyP$-subgroup
    of $\theta(t)$.
    Now $\theta(A) = \cz{\theta(t)}{A}$ so Theorem~\ref{p:2} implies that
    $\gen{ \cz{[\theta(a),t]}{t} }{ a \in A\nonid }$ is a $\fancyP$-group.
    Then (a) implies that $\theta(t)$ is a $\fancyP$-group.
\end{proof}

\begin{Lemma}\label{gor:6}
    Suppose that $A$ acts on the $K$-group $H$,
    that $t \in A\nonid$ and that $t$ acts trivially on $\compsol{H}$.
    Then $t$ acts trivially on $\compsol{M}$
    whenever $M$ is an $A\cz{H}{A}$-invariant subgroup of $H$.
\end{Lemma}
\begin{proof}
    Assume false and choose $K_{0} \in \compsol{M}$
    with $K_{0} \not= K_{0}^{t}$.
    Set $K = \listgen{K_{0}^{A}} \in \compasol{M}$.
    Now $[K,t]$ is an $A$-invariant nonsolvable normal subgroup of $K$
    so $K = [K,t] \leq [H,t]$.
    Since $[H,t] \normal H$ we have
    $\compsol{[H,t]} \subseteq \compsol{H}$,
    hence we may assume that $H = [H,t]$.
    Passing to the quotient $H/\sol{H}$ we may also assume that $\sol{H} = 1$.
    Then $\compsol{H} = \compp{H}$ and $\cz{H}{\layerr{H}} = 1$.

    Since $H = [H,t]$ and $t$ acts trivially on $\compp{H}$ it follows that
    every component of $H$ is normal in $H$.
    As $\cz{H}{\layerr{H}} = 1$,
    the Schreier Property implies that $H/\layerr{H}$ is solvable.
    Now $K$ is perfect so $K \leq \layerr{H}$ and then Lemma~\ref{lglob:6}
    implies there exists $K^{*}$ with $K \leq K^{*} \in \comp{A}{H}$.
    Hence we may assume that $K^{*} = H$,
    so that $H$ is $A$-simple.

    Without loss, $\cz{A}{H} = 1$.
    Set \[
        A_{\infty} = \ker(A \longrightarrow \sym{\compp{H}}),
    \]
    so that $t \in A_{\infty}$.
    Since $M$ is $A\cz{H}{A}$-invariant and nonsolvable,
    it follows from \ICiteLemmasAQSsixANDAQSeight\ that either
    $M = \cz{H}{B}$ for some $B \leq A$ with $B \cap A_{\infty} = 1$
    or $M \leq \cz{H}{A_{\infty}}$.
    The second possibility does not hold since $t \in A_{\infty}$
    and $K = [K,t] \leq M$.
    Thus the first possibility holds.
    \ICiteLemmaAQSfive\ implies that $M$ is $A$-simple.
    Since $K \in \comp{A}{M}$ we have $K = M$,
    so $K = \cz{H}{B}$.
    Then the components of $K$ correspond to the orbits
    of $B$ on $\compp{H}$.
    Since $t$ is trivial $\compp{H}$ it normalizes each orbit of $B$
    and hence each component of $K$.
    This contradiction completes the proof.
\end{proof}

\begin{Lemma}\label{gor:7}
    Suppose that $A$ acts coprimely on the group $H$.
    Let $t \in A$ and suppose that $t$ acts trivially
    on $\compsol{H}$.
    If $S$ is a simple section of $\cz{[H,t]}{t}$
    then there exists $L \in \compsol{H}$ with
    $\card{S} < \card{L/\sol{L}}$.
\end{Lemma}
\begin{proof}
    Coprime Action implies that $[H,t,t] = [H,t]$ and as $[H,t] \normal H$
    we have $\compsol{[H,t]} \subseteq \compsol{H}$
    so we may suppose that $H = [H,t]$.
    We may also pass to the quotient $H/\sol{H}$
    to suppose that $\sol{H} = 1$.

    Since $t$ acts trivially on $\compp{H}$ and $H = [H,t]$
    it follows that every component of $H$ is normal in $H$.
    Let $L_{1},\ldots,L_{n}$ be the components of $H$,
    so $\layerr{H} = L_{1} \times\cdots\times L_{n}$.
    As $\sol{H} = 1$ we have $\cz{H}{\layerr{H}} = 1$ and the
    Schreier Property implies that $H/\layerr{H}$ is solvable.
    In particular,
    $\cz{H}{t}/\cz{\layerr{H}}{t}$ is solvable so
    $S$ is isomorphic to a simple section of $\cz{\layerr{H}}{t}$.
    Now $\cz{\layerr{H}}{t} = \cz{L_{1}}{t} \times\cdots\times \cz{L_{n}}{t}$.
    Thus $S$ is isomorphic to a simple section of $\cz{L_{i}}{t}$
    for some $i$.
    Note that $\cz{L_{i}}{t} \not= L_{i}$ since otherwise,
    as $L_{i} \normal H\listgen{t}$ and $H = [H,t]$,
    we would have $L_{i} \leq \zz{H}$.
    Thus $\card{S} < \card{L_{i}}$ and the proof is complete.
\end{proof}

\begin{proof}[Proof of Theorem~\ref{gor:2}]
    Assume false and let $\theta$ be a minimal counterexample.
    By Lemma~\ref{gor:6}, the hypotheses of Lemma~\ref{gor:5}
    are satisfied.
    Note that a group $X$ is nonsolvable
    if and only if $\compsol{X} \not= \emptyset$.
    By the Solvable Signalizer Functor Theorem,
    there exists a pair $(b,K)$ with
    $b \in A\nonid$ and $K \in \compsol{\theta(b)}$.
    Choose such a pair with $\card{K/\sol{K}}$ maximal.
    By Lemma~\ref{gor:5} there exists $a,t \in A\nonid$ such that
    $K/\sol{K}$ is isomorphic to a simple section
    of $\cz{[\theta(a),t]}{t}$.
    Lemma~\ref{gor:7} implies $\card{K/\sol{K}} < \card{L/\sol{L}}$
    for some $L \in \compsol{\theta(a)}$.
    This contradicts the choice of $(b,K)$
    and completes the proof.
\end{proof}

%% file: bib.tex
\bibliographystyle{amsplain}

%% file: root.bbl
\begin{thebibliography}{9}


    \bibitem{I}     P. Flavell,
                    {\em Automorphisms of $K$-groups I,}
                    Preprint: http://arxiv.org/abs/1609.01969

    \bibitem{F2}    P. Flavell,
                    {\em A new proof of the Solvable Signalizer Functor Theorem,}
                    J. Algebra {\bf 398} (2014) 350-363

    \bibitem{GL}    D. Gorenstein and R. Lyons,
                    {\em Nonsolvable Signalizer Functors on Finite Groups,}
                    Proc. London Math. Soc. {\bf 35} (3) (1977) 1-33


    \bibitem{McB1}  P.P. McBride,
                    {\em Near solvable signalizer functors on finite groups,}
                    J. Algebra {\bf 78}(1) (1982) 181-214

    \bibitem{McB2}  P.P. McBride,
                    {\em Nonsolvable signalizer functors on finite groups,}
                    J. Algebra {\bf 78}(1) (1982) 215-238

\end{thebibliography}
